\documentclass[a4paper,12pt]{article}
\usepackage[a4paper,margin=1.15in]{geometry}

\usepackage{amsmath,amssymb,amsthm}
\usepackage[scr=boondoxo]{mathalpha}
\usepackage{xcolor}
\usepackage{xspace}
\usepackage{enumitem}
\usepackage[colorlinks=true,allcolors=blue]{hyperref}
\usepackage{aliascnt}
\usepackage[nameinlink,capitalise]{cleveref}

\newif\ifdraft
\draftfalse
\long\def\draft#1{\ifdraft\begingroup\color{teal}#1\endgroup\else\ignorespaces\fi}

\providecommand{\subjclass}[2][]{\textbf{\textit{#1 Mathematics Subject Classification: }} #2}
\providecommand{\keywords}[1]{\textbf{\textit{Keywords: }} #1}

\newtheorem{theorem}{Theorem}[section]
\newaliascnt{proposition}{theorem}
\newtheorem{proposition}[proposition]{Proposition}
\aliascntresetthe{proposition}
\newaliascnt{lemma}{theorem}
\newtheorem{lemma}[lemma]{Lemma}
\aliascntresetthe{lemma}
\newaliascnt{corollary}{theorem}
\newtheorem{corollary}[corollary]{Corollary}
\aliascntresetthe{corollary}
\newaliascnt{remark}{theorem}
\newtheorem{remark}[remark]{Remark}
\aliascntresetthe{remark}
\newaliascnt{definition}{theorem}
\newtheorem{definition}[definition]{Definition}
\aliascntresetthe{definition}
\newaliascnt{example}{theorem}
\newtheorem{example}[example]{Example}
\aliascntresetthe{example}

\newcommand{\mathmode}[1]{\ensuremath{#1}\xspace}

\newcommand{\aset}[1]{\mathmode{\mathcal{#1}}}
\newcommand{\Sset}{\aset{S}}

\newcommand{\Mset}{\aset{M}}
\newcommand{\Tset}{\mathmode{\mathscr{T}}}
\newcommand{\G}{\mathmode{G}}
\renewcommand{\H}{\mathmode{H}}
\newcommand{\K}{\mathmode{K}}
\newcommand{\Q}{\mathmode{Q}}

\newcommand{\R}{\mathmode{R}}

\newcommand{\uR}{\mathmode{\R^{\times}}}
\newcommand{\F}{\mathmode{\mathbb{F}}}

\newcommand{\gA}[2]{\mathmode{#1#2}}

\newcommand{\RH}{\gA{\R}{\H}}
\newcommand{\RK}{\gA{\R}{\K}}
\newcommand{\RQ}{\gA{\R}{\Q}}

\newcommand{\FH}{\gA{\F}{\H}}
\newcommand{\FK}{\gA{\F}{\K}}
\newcommand{\FQ}{\gA{\F}{\Q}}

\newcommand{\modu}[1]{\mathmode{\mathcal{#1}}}
\newcommand{\mA}{\modu{A}}
\newcommand{\mM}{\modu{M}}
\newcommand{\mN}{\modu{N}}
\newcommand{\mP}{\modu{P}}

\newcommand{\mV}{\modu{V}}
\newcommand{\mW}{\modu{W}}

\newcommand{\bidemp}[1]{\mathmode{\mathbf{#1}}}

\newcommand{\bee}{\bidemp{e}}

\newcommand{\hH}{\mathmode{{\widehat{\H}}}}
\newcommand{\hK}{\mathmode{{\widehat{\K}}}}
\newcommand{\hQ}{\mathmode{{\widehat{\Q}}}}

\DeclareMathOperator{\Hom}{Hom}
\DeclareMathOperator{\Aut}{Aut}
\DeclareMathOperator{\Ind}{Ind}
\DeclareMathOperator{\Span}{Span}
\DeclareMathOperator{\Fun}{Fun}
\newcommand{\Fourier}{\mathmode{\mathcal{F}}}
\newcommand{\gr}{\mathrm{gr}}
\newcommand{\Homgr}{\Hom^{\gr}}

\title{Graded bimodules over split group algebras:\\stabilizers and Peirce decompositions}
\author{Felipe de Mattos Chafik Hindi\footnote{Department of Mathematics, UFSCar, Brazil, \texttt{felipehindi@gmail.com}}\\
Waldeck Sch\"utzer\footnote{Department of Mathematics, UFSCar, Brazil, \texttt{waldeck@ufscar.br}}}
\date{September 25, 2026}

\begin{document}
\maketitle

\begin{abstract}
Let $\R$ be a nonzero indecomposable commutative ring, $\G$ an arbitrary
group, and $\H,\K\leq\G$ finite abelian
subgroups such that $\RH\cong\R^{|\H|}$ and $\RK\cong\R^{|\K|}$.  No
semisimplicity hypothesis is imposed on $\R$ or on these group algebras.  We give a
self-contained treatment of the classification developed in our work on
incidence algebras: up to graded isomorphism, all $\G$-graded $(\RH,\RK)$-bimodules
are classified by the double
cosets $\H g\K$, the characters of their stabilizers, and arbitrary
$\R$-modules.  This extends the bimodule classification of
Santulo, Souza, and Yasumura \cite[Section~4]{SSY2020}, where the grading group
is abelian and the controlling subgroup is $\H\cap\K$.  Here it is replaced
by the stabilizer attached to each double coset.  Over a field satisfying the
same splitting assumptions, we determine, on each double coset, the essential
image of the induction functor of Dos Santos and Yasumura on Yetter-Drinfel'd
modules and answer negatively the question following
\cite[Corollary~3.8]{DSY2025}.  We also compute the Peirce components explicitly
and characterize multiplicity-free bimodules by pairwise disjoint sets of
extensions of stabilizer characters.  We give a finite Fourier formula for the
Peirce multiplicities.
\end{abstract}



\section{Introduction}

Let $\R$ be a nonzero indecomposable commutative ring, $\G$ a group, and
$\H,\K\leq\G$ finite abelian subgroups such that
$\RH\cong\R^{|\H|}$ and $\RK\cong\R^{|\K|}$.  The splitting hypothesis
does not require semisimplicity.  These product decompositions imply
that $\RH$ and $\RK$ are semisimple if and only if $\R$ is semisimple (see
the general form of Maschke's theorem \cite[Theorem~3.4.7]{Polcino-Sehgal}).
We do not assume that $\R$ is semisimple.  If
$\mM=\bigoplus_{g\in\G}\mM^g$ is a $\G$-graded $(\RH,\RK)$-bimodule, then
$h\mM^gk\subseteq\mM^{hgk}$.
Thus $\Q=\H\times\K^{\mathrm{op}}$ acts on the set of degrees by
$(h,k^{\mathrm{op}})\cdot g=hgk$.  This is a special case of the theory of
modules graded by $G$-sets \cite{NRV1990}.  See also
\cite{deFrancaSviridova2024} for graded bimodules over finite-dimensional
algebras.

This paper arose from \cite{HindiSchutzer2026Incidence}, where such bimodules
occur in the graded Peirce decomposition of an incidence algebra.
Section~7.2 of that paper establishes the general bimodule classification,
the graded Hom formula, the Peirce formula, and the multiplicity-free
criterion.  We give a self-contained treatment of these results and develop
their consequences for the induction functor on Yetter-Drinfel'd modules,
finite Fourier transforms, and multiplicity-free gradings.

The $\Q$-orbits are the double cosets $\H g\K$, with stabilizers
\[
 \Sset_g(\H,\K)=\{(h,k)\in\H\times\K:hgk=g\}
 \cong \H\cap g\K g^{-1}.
\]
We prove that each $\R\Sset_g$ is again split and classify the graded bimodules
by the double cosets, the characters of their stabilizers, and arbitrary
$\R$-modules.  For a fixed double coset and character, the corresponding
construction preserves all $\R$-module homomorphisms.  Thus, over a general
indecomposable ring, the classification cannot be reduced to ranks.

In \cite[Section~4]{SSY2020}, Santulo, Souza, and Yasumura
work over a field with an abelian grading group.  They
obtain a decomposition of finite-dimensional graded bimodules controlled by
characters of $\H\cap\K$, together with an isomorphism criterion
\cite[Corollary~23 and Lemma~26]{SSY2020}.  Their Remark~24 exhibits a
nonabelian grading for which that description fails.  Here the fixed subgroup
$\H\cap\K$ is replaced by a stabilizer for each double coset, and scalar
multiplicities are replaced by coefficient $\R$-modules.  Thus the
classification retains the abelian result while also accounting for their
free-product example as a double coset with trivial stabilizer.  See
\cref{rem:SSY-free-product}.

For the bimodules in graded incidence algebras,
\cite[Theorem~2]{SSY2020} proves that the characters in the
abelian decomposition are pairwise distinct.  Our multiplicity-free criterion
recovers that condition when all stabilizers coincide and replaces it in
general by the requirement that characters on different stabilizers disagree
on their intersections.  This concerns individual bimodules.  It does not
settle the incidence-algebra reconstruction problem posed in
\cite[Section~7.4]{SSY2020}, which also involves compatibility of products
between Peirce components.  The classification and realizability conditions
incorporating these products are given in
\cite[Section~9]{HindiSchutzer2026Incidence}.

Dos Santos and Yasumura prove, for abelian $\G$, that their induction functor
on Yetter-Drinfel'd modules is faithful and essentially surjective.  They ask whether this
remains true for nonabelian $\G$
\cite[Corollary~3.8 and the following question]{DSY2025}.  For a double coset $\H g\K$, consider the inclusion
$\Delta C_{\H\cap\K}(g)\leq \Sset_g$, where
$\Delta B=\{(b,b^{-1}):b\in B\}$.  The functor is essentially surjective on
that double coset precisely when this is an equality.  Taking $(\G,\H,\K)=(A_4,V_4,V_4)$ gives a negative answer to
their question.

Finally, we determine all Peirce components of the classified bimodules and
study the sets of extensions to $\Q$ of stabilizer characters.  Their finite Fourier transforms give restrictions on Peirce multiplicities and on
decompositions that are multiplicity-free.  We determine how two stabilizers
can occur in such a decomposition and obtain exactly $24$ full-support
multiplicity-free gradings up to graded isomorphism when
$(\G,\H,\K)=(S_4,V_4,V_4)$, with $V_4$ the normal Klein four subgroup.
The resulting coset partitions are related to
\cite[Theorems~2 and~3]{AkmanSissokho2025}.

\Crefrange{sec:stabilizers}{sec:classification} give the classification,
\Cref{sec:YD} treats the induction functor on Yetter-Drinfel'd modules,
\cref{sec:Fourier-Peirce} treats the Peirce decomposition and Fourier
transform, and \cref{sec:multiplicity-free} treats the multiplicity-free case.
The classical connection between group gradings, smash products, and group
actions developed by Cohen and Montgomery \cite{CM84} also motivates the
smash-product question in \cref{sec:questions}.  We close with further
questions suggested by these descriptions.

\draft{\par\medskip\noindent\textbf{How to use this verification layer.}
The black text is the paper.  The teal text is a checking aid and supplies no
additional hypothesis.  For each proof, record the hypotheses in force,
check every balancing relation before using a formula on pure tensors, check
the degree of every generator and how the group action transports it, and
check separately that every asserted direct sum spans and has zero
intersection.  For an ``if and only if,'' check the two implications
separately.

The equivalence relation is strict graded isomorphism: degree $g$ must map to
degree $g$.  Shifts, conjugations, and automorphisms of $\G$ are not allowed.
The background used without proof is basic commutative algebra, elementary
group theory, and, in one remark, the normal form theorem for free products.\par\medskip}

\section{Stabilizers and split group algebras}
\label{sec:stabilizers}

Throughout, $\R$ is a nonzero indecomposable commutative ring with identity,
that is, its only idempotents are $0$ and $1$, and $\G$ is a group.  All
algebra homomorphisms are unital.  For a finite group $L$, we write $\widehat L=\Hom(L,\R^\times)$ for its
group of $\R$-valued characters.  For a set $X$, let
$\Fun(X,\R)$ denote the $\R$-module of functions $X\to\R$.  We call the group algebra $\R L$
\emph{split} when $\R L\cong\R^{|L|}$ as $\R$-algebras.  Here and below,
``split'' refers only to this product decomposition.  No semisimplicity
assumption is in force.  Let $\H,\K\leq\G$ be finite abelian subgroups such that
$\RH\cong\R^{|\H|}$ and $\RK\cong\R^{|\K|}$.  All bimodules are unital and $\R$-central, that is, $rm=mr$ for
$r\in\R$.  Put $\Q=\H\times\K^{\mathrm{op}}$.  We view an
$(\RH,\RK)$-bimodule as a left $\RQ$-module by
$(h,k^{\mathrm{op}})\cdot m=hmk$.  Accordingly, $\Q$ acts on the set of degrees by
\begin{equation}
 (h,k^{\mathrm{op}})\cdot g=hgk.
 \label{eq:Q-action}
\end{equation}
The factor $\K^{\mathrm{op}}$ records the right action.  Since $\K$ is abelian,
we identify it with $\K$ when convenient and write $(h,k)$ for
$(h,k^{\mathrm{op}})$.

A $\G$-graded $(\RH,\RK)$-bimodule is an $\R$-module
$\mM=\bigoplus_{g\in \G}\mM^g$ such that
$h\mM^g k\subseteq \mM^{hgk}$ for $h\in\H$, $k\in\K$, and $g\in\G$.
Since $h$ and $k$ are units, this inclusion is an equality.  The support of
$\mM$ is $\{g\in\G:\mM^g\neq0\}$.  A graded homomorphism preserves every
homogeneous component.  All graded isomorphisms below are understood in this
sense.

For $g\in \G$ define
\[
 \Sset_g=\Sset_g(\H,\K)=\{(h,k)\in \H\times \K: hgk=g\}.
\]
Thus $\Sset_g$ is the stabilizer of $g$ for the action \eqref{eq:Q-action}.

\draft{\par\smallskip\noindent\textbf{Checking the opposite-group convention.}
In $\Q=\H\times\K^{\mathrm{op}}$,
\[
 (h_1,k_1^{\mathrm{op}})(h_2,k_2^{\mathrm{op}})
 =(h_1h_2,(k_2k_1)^{\mathrm{op}}).
\]
Successive action on $m$ gives
\[
 h_1(h_2mk_2)k_1=h_1h_2m(k_2k_1),
\]
which is exactly the action of this product.  The same calculation on a
degree gives $h_1h_2gk_2k_1$.  This is why the opposite group is required.
Because $\K$ is abelian, later formulas may write $(h,k)$ without changing
the orientation.

Also, $h\mM^gk\subseteq\mM^{hgk}$ is an equality.  If
$y\in\mM^{hgk}$, then $h^{-1}yk^{-1}\in\mM^g$ and
$y=h(h^{-1}yk^{-1})k$.  Thus $\Q$ really permutes the homogeneous
components.\par\smallskip}

\begin{lemma}
\label{lem:stabilizer-basic}
The $\Q$-orbit of $g$ is the double coset $\H g \K$.  Moreover, the coordinate
projections restrict to injective maps on $\Sset_g$, and
\[
 \Sset_g\cong \H\cap g\K g^{-1}\cong \K\cap g^{-1}\H g.
\]
If $t\in \H g \K$, then $\Sset_t=\Sset_g$.
\end{lemma}

\begin{proof}
The orbit is $\H g\K$.  If $(h,k)\in\Sset_g$, then
$h=gk^{-1}g^{-1}$ and $k=g^{-1}h^{-1}g$, so each coordinate projection is
injective.  Their images are $\H\cap g\K g^{-1}$ and
$\K\cap g^{-1}\H g$, respectively.

Let $t=h_0gk_0$.  Since $\H$ and $\K$ are abelian,
\[
 htk=t
 \iff hh_0gk_0k=h_0gk_0
 \iff hgk=g.
\]
Hence $\Sset_t=\Sset_g$.
\end{proof}

\draft{\par\smallskip\noindent\textbf{\Cref{lem:stabilizer-basic}, line by line.}
The orbit consists of all $(h,k^{\mathrm{op}})\cdot g=hgk$, hence is exactly
$\H g\K$.  From $hgk=g$ we obtain
\[
 h=gk^{-1}g^{-1}\quad\text{and}\quad k=g^{-1}h^{-1}g.
\]
Therefore either coordinate determines the other.  The possible first
coordinates are precisely $\H\cap g\K g^{-1}$, and the possible second
coordinates are $\K\cap g^{-1}\H g$.

Now let $t=h_0gk_0$.  The equality $htk=t$ is
$hh_0gk_0k=h_0gk_0$.  Multiply by $h_0^{-1}$ on the left and $k_0^{-1}$
on the right.  Commutativity inside $\H$ and inside $\K$ turns the result
into $hgk=g$.  Every step is reversible.  Hence stabilizers along a double
coset are literally equal, rather than merely conjugate or abstractly
isomorphic.\par\smallskip}

Write $\Sset_C=\Sset_g$ for $C=\H g \K$.

\begin{lemma}
\label{lem:split-stabilizers}
For every subgroup $S\leq \Q$ one has
$\R S\cong \R^{|S|}$.
In particular, this holds for every stabilizer $\Sset_C$.  Moreover, restriction of
characters
\[
 \operatorname{res}^\Q_S:\hQ=\Hom(\Q,\R^\times)
 \longrightarrow \widehat S=\Hom(S,\R^\times)
\]
is surjective, and every character of $S$ has exactly $[\Q:S]$ extensions to
$\Q$.
\end{lemma}

\begin{proof}
Put $n=|\Q|$.  By the standing splitting assumptions,
\[
 \RQ\cong\RH\otimes_\R(\RK)^{\mathrm{op}}\cong\R^n.
\]
We first note that $n\in\uR$.  Indeed, for every maximal ideal $\mathfrak m$
of $\R$, base change gives
\[
 (\R/\mathfrak m)\Q\cong(\R/\mathfrak m)^n.
\]
If $n=0$ in $\R/\mathfrak m$, then
$z=\sum_{q\in\Q}q$ is nonzero in $(\R/\mathfrak m)\Q$ but
$z^2=nz=0$, whereas the product of fields $(\R/\mathfrak m)^n$ has no
nonzero nilpotent elements.  Hence $n\notin\mathfrak m$.  Since this holds for
every maximal ideal, $n$ is a unit in $\R$.  As $n=|S|[\Q:S]$, also
$|S|\in\uR$.

Let $B=\R S\subseteq\RQ\cong\R^n$.  Since $|S|$ is a unit, the element
\[
 E=|S|^{-1}\sum_{s\in S}s^{-1}\otimes s\in B\otimes_\R B
\]
satisfies $\mu(E)=1$ and $(b\otimes1)E=E(1\otimes b)$ for every $b\in B$.
Thus it is a separability idempotent for $B$.  Let $f:B\to\R$ be the
restriction of a coordinate projection of $\R^n$.  Writing
$E=\sum_i x_i\otimes y_i$, put
\[
 e=\sum_i x_i f(y_i).
\]
Applying $\operatorname{id}\otimes f$ to $(b\otimes1)E=E(1\otimes b)$ gives
$be=f(b)e$.  Also $f(e)=f(\mu(E))=1$, and taking $b=e$ gives $e^2=e$.
Therefore $eB=\R e\cong\R$ and $\ker f=(1-e)B$, so
\[
 B\cong\R\times(1-e)B.
\]
Under the embedding $B\subseteq\R^n$, the coordinates of $e$ are idempotents
of $\R$, hence are $0$ or $1$ because $\R$ is indecomposable.  If
$(1-e)B\neq0$, then $1-e$ has a coordinate equal to $1$.  Restriction of that
coordinate is therefore a unital $\R$-algebra map $(1-e)B\to\R$.  If
$p:B\to(1-e)B$ is the projection, then $(p\otimes p)(E)$ satisfies the two
defining identities for a separability idempotent of $(1-e)B$.  Thus the same
argument splits another copy of $\R$ from $(1-e)B$.  Inductively, after
splitting off $r$ copies we have
\[
 B\cong\R^r\times B_r.
\]
As an $\R$-module, $B$ is free of rank $|S|$.  Hence, after localization at
any maximal ideal $\mathfrak m$, $(B_r)_{\mathfrak m}$ is free of rank
$|S|-r$ (a finitely generated projective module over a local ring is free).  Thus $B_r\neq0$ whenever $r<|S|$, and the preceding coordinate
argument applies again.  For $r=|S|$ the remaining factor has zero localization
at every maximal ideal and is therefore zero: if $0\neq x\in B_r$, choosing a
maximal ideal containing its annihilator
$\{a\in\R:ax=0\}$ would make $x/1\neq0$ in $(B_r)_{\mathfrak m}$.  Thus
\[
 \R S\cong\R^{|S|}.
\]

We also need the quotient.  Since $\Q$ is abelian, $S$ is normal.  Let $I$
be the kernel of the natural map $\RQ\to\R(\Q/S)$.  Under $\RQ\cong\R^n$, every ideal is the product of its coordinate
ideals, so write $I=I_1\times\cdots\times I_n$.  Then
\[
 \R(\Q/S)\cong\prod_{i=1}^n\R/I_i.
\]
The left side is free over $\R$, so each $\R/I_i$, being an $\R$-direct
summand, is projective.  Hence the exact sequence
$0\to I_i\to\R\to\R/I_i\to0$ splits, so $I_i$ is a direct summand ideal of
$\R$.  Such an ideal is $e_i\R$ for an idempotent $e_i\in\R$.  Indeed, take
$e_i$ to be the image of $1$ under the projection $\R\to I_i$.
Indecomposability gives $I_i=0$ or $I_i=\R$.
Comparing $\R$-ranks yields
\[
 \R(\Q/S)\cong\R^{[\Q:S]}.
\]

For a finite group $T$ with $\R T\cong\R^{|T|}$, a character
$T\to\R^\times$ extends uniquely by $\R$-linearity to an $\R$-algebra map
$\R T\to\R$, and restriction to $T$ gives the inverse correspondence.  If
$f:\R^m\to\R$ is an $\R$-algebra map and $e_1,\ldots,e_m$ are the standard
coordinate idempotents, then the $f(e_i)$ are orthogonal idempotents with sum
$1$.  Indecomposability of $\R$ forces exactly one of them to be $1$ and all
others to be $0$.  Hence $f$ is the corresponding coordinate projection.
Thus there are exactly $|T|$ characters.  Therefore
$|\hQ|=|\Q|$, $|\widehat S|=|S|$, and
$|\widehat{\Q/S}|=[\Q:S]$.  The kernel of
$\hQ\to\widehat S$ is $\widehat{\Q/S}$, so the image has $|S|$ elements and
is all of $\widehat S$.  Every fiber has $[\Q:S]$ elements.
\end{proof}

\draft{\par\medskip\noindent\textbf{\Cref{lem:split-stabilizers}, detailed verification.}
This is a ring argument, not a field diagonalization.
\begin{enumerate}[label=\textup{(\arabic*)}]
\item The splitting assumptions give
\[
 \RQ\cong\RH\otimes_\R(\RK)^{\mathrm{op}}
 \cong\R^{|\H||\K|}=\R^n,\quad\text{where }n=|\Q|.
\]
The elements of $\Q$ form a basis after reduction modulo any maximal ideal
$\mathfrak m$, so $z=\sum_q q$ remains nonzero.  There are exactly $n$
pairs $(q,r)$ with $qr=u$ for each $u\in\Q$, whence $z^2=nz$.  If $n$
vanished in $\R/\mathfrak m$, this would be a nonzero nilpotent in a product
of fields.  Thus $n$ lies in no maximal ideal and is a unit.  Since
$n=|S|[\Q:S]$, both integer factors are units.  For example
$|S|^{-1}=[\Q:S]n^{-1}$.

\item For $B=\R S$,
\[
 E=|S|^{-1}\sum_{s\in S}s^{-1}\otimes s
\]
satisfies $\mu(E)=1$.  To check
$(t\otimes1)E=E(1\otimes t)$ for $t\in S$, substitute
$u=st^{-1}$ in one sum.  Linearity then gives the identity for every
$b\in B$.

\item Given a coordinate map $f:B\to\R$, set
$e=\sum_i x_if(y_i)$.  Applying $\operatorname{id}\otimes f$ to the
separability identity gives $be=f(b)e$.  Also
\[
 f(e)=\sum_i f(x_i)f(y_i)
 =f\left(\sum_i x_iy_i\right)=f(1)=1,
\]
so taking $b=e$ gives $e^2=e$.  If $f(b)=0$, then $be=0$ and
$b=(1-e)b$.  The reverse inclusion in $\ker f=(1-e)B$ is immediate.
Moreover $eB=\R e$.  Hence $B\cong\R\times(1-e)B$ as algebras.

\item Every coordinate of $e\in\R^n$ is an idempotent, hence is $0$ or
$1$.  If $(1-e)B\neq0$, some coordinate of $1-e$ is $1$, so that
coordinate restricts to another unital map to $\R$.  Applying the projection
$B\to(1-e)B$ to both tensor factors of $E$ gives a separability idempotent
for the complementary factor.  Therefore the splitting step repeats.

\item After $r$ steps, $B\cong\R^r\times B_r$.  Since $B$ is free of rank
$|S|$, localization makes $(B_r)_{\mathfrak m}$ free of rank $|S|-r$.
Thus $B_r\neq0$ when $r<|S|$.  At $r=|S|$ all localizations vanish.  If
$x\neq0$ remained, a maximal ideal containing $\operatorname{ann}(x)$ would
make $x/1\neq0$, a contradiction.  Therefore $\R S\cong\R^{|S|}$.

\item The quotient requires a separate check.  Since $\Q$ is abelian, $S$
is normal.  Each quotient factor $\R/I_i$ is projective, so
$0\to I_i\to\R\to\R/I_i\to0$ splits and $I_i=e_i\R$ for an idempotent
$e_i$.  Indecomposability gives $I_i=0$ or $\R$.  The cosets of $S$ form an
$\R$-basis of $\R(\Q/S)$, so exactly $[\Q:S]$ coordinate factors survive.

\item A character $T\to\R^\times$ extends uniquely to a unital algebra map
$\R T\to\R$.  Conversely, such a map restricts to a character because group
elements are units.  A unital map $\R^m\to\R$ sends the coordinate
idempotents to orthogonal idempotents summing to $1$.  Indecomposability forces
exactly one image to be $1$.  Thus a split group algebra has $|T|$
characters.  The kernel of restriction $\hQ\to\widehat S$ is
$\widehat{\Q/S}$, of size $[\Q:S]$.  Therefore the image has
$|\Q|/[\Q:S]=|S|$ elements and equals $\widehat S$.  Every fiber is a coset
of the kernel.
\end{enumerate}
The uses of the hypotheses are now visible: finiteness supplies averaging and
ranks, abelianness supplies normality and commutative character theory,
splitting starts the coordinate argument, and indecomposability turns
idempotents into zeros and ones.\par\medskip}

For $\lambda\in\widehat S$, let
\begin{equation}
 \bee_\lambda=\frac{1}{|S|}\sum_{s\in S}\lambda(s^{-1})s\in \R S.
 \label{eq:Fourier-idempotent}
\end{equation}
For later use we record the character orthogonality relations.  If
$T\leq\Q$ and $\psi\in\widehat T$, then
\begin{equation}
 \sum_{t\in T}\psi(t)
 =\begin{cases}
   |T|,&\psi=1,\\
   0,&\psi\neq1.
  \end{cases}
 \label{eq:character-sum}
\end{equation}
Indeed, if $a=|T|^{-1}\sum_{t\in T}t$, then $ta=a$ for every $t\in T$
and $a^2=a$.  For any $\psi\in\widehat T$, the element $\psi(a)$ is an
idempotent of the indecomposable ring $\R$, hence is $0$ or $1$.  If
$\psi(a)=1$, then $\psi(t)=\psi(ta)=1$ for every $t\in T$.  Conversely,
the trivial character takes $a$ to $1$.  This proves
\eqref{eq:character-sum}.  Hence, for $\alpha,\beta\in\widehat T$,
\begin{equation}
 \frac1{|T|}\sum_{t\in T}\alpha(t)\beta(t^{-1})
 =\delta_{\alpha,\beta}.
 \label{eq:character-orthogonality}
\end{equation}
Since $|\widehat T|=|T|$, the matrix
$X=(\alpha(t))_{\alpha\in\widehat T,\ t\in T}$ is square, and
\eqref{eq:character-orthogonality} says that its inverse is
$|T|^{-1}(\beta(t^{-1}))_{t\in T,\ \beta\in\widehat T}$.  Thus the characters
of $T$ form an $\R$-basis of $\Fun(T,\R)$.  Applying
\eqref{eq:character-orthogonality} to the elements \eqref{eq:Fourier-idempotent}
shows that the $\bee_\lambda$ are pairwise orthogonal and sum to $1$.
Moreover $\R S\bee_\lambda=\R\bee_\lambda\cong\R$.  Any decomposition
of $\bee_\lambda$ into two nonzero orthogonal idempotents would therefore
give a nontrivial idempotent of $\R$.  Hence $\bee_\lambda$ is primitive.  Thus
$(\bee_\lambda)_{\lambda\in\widehat S}$ is the complete set of primitive
orthogonal idempotents of $\R S$, and
$s\bee_\lambda=\lambda(s)\bee_\lambda$ for every $s\in S$.

\draft{\par\smallskip\noindent\textbf{Fourier idempotents and orthogonality over $\R$.}
The averaging element $a=|T|^{-1}\sum_t t$ satisfies $a^2=a$ and $ta=a$.
Its image under the trivial character is $1$, so under
$\R T\cong\R^{|T|}$ it is the coordinate idempotent belonging to that
character.  Every other coordinate character sends it to $0$.  Multiplying
these values by $|T|$ gives \eqref{eq:character-sum}.

Apply that formula to $\psi=\alpha\beta^{-1}$.  Since
$\beta(t^{-1})=\beta(t)^{-1}$, one obtains
\eqref{eq:character-orthogonality}.  In matrix form this says
\[
 \sum_{t\in T}\alpha(t)
 \bigl(|T|^{-1}\beta(t^{-1})\bigr)=\delta_{\alpha,\beta},
\]
which checks the displayed inverse of the character table.  Consequently the
characters are an $\R$-basis of all functions $T\to\R$.

For the Fourier idempotents, direct multiplication followed by this same
orthogonality gives
\[
 \bee_\lambda\bee_\mu=\delta_{\lambda,\mu}\bee_\lambda
 \quad\text{and}\quad \sum_{\lambda\in\widehat S}\bee_\lambda=1
 \quad\text{and}\quad s\bee_\lambda=\lambda(s)\bee_\lambda.
\]
Thus every $\R S$-module $N$ decomposes as
$N=\bigoplus_\lambda\bee_\lambda N$, even when $N$ is not free, projective,
or finitely generated.  No field-only semisimplicity assertion is being used.\par\smallskip}

\section{Weight spaces and classification}
\label{sec:classification}

Choose $g_C\in C$ for each $C\in\H\backslash\G/\K$ and put
$\Sset_C=\Sset_{g_C}$.  Define
\[
 \Tset_{\H,\K}(\G)
 =\coprod_{C\in \H\backslash \G/\K}\widehat{\Sset_C}.
\]
We call a pair $\theta=(C,\lambda)$ a \emph{type}.  If $\mM$ is a graded bimodule and
$\lambda\in\widehat{\Sset_g}$, its $\lambda$-weight space in degree $g$ is
\[
 (\mM^g)_\lambda
 =\{m\in\mM^g:s\cdot m=\lambda(s)m\text{ for every }s\in\Sset_g\}
 =\bee_\lambda\mM^g.
\]
Indeed, if $s\cdot m=\lambda(s)m$ for every $s\in\Sset_g$, then
$\bee_\lambda m=m$ by \eqref{eq:Fourier-idempotent}.  Conversely,
$s\bee_\lambda=\lambda(s)\bee_\lambda$ shows that every element of
$\bee_\lambda\mM^g$ has weight $\lambda$.

If $\mP$ is an $\R$-module and $\lambda\in\widehat{\Sset_C}$, write $\mP_\lambda$ for
$\mP$ endowed with the $\R\Sset_C$-module structure $s\cdot p=\lambda(s)p$.
Define the induced graded bimodule
\[
 \mA_\theta(\mP)
   =\RQ\otimes_{\R\Sset_C}\mP_\lambda,
\]
with grading
\begin{equation}
 \deg(q\otimes p)=q\cdot g_C.
 \label{eq:induced-bimodule-grading}
\end{equation}
The balancing relation is homogeneous because $\Sset_C$ fixes $g_C$.

Following \cite{HindiSchutzer2026Incidence}, a graded $(\RH,\RK)$-bimodule is
\emph{atomic of weight $\lambda$} if it has the form $\RH m\RK$, where $m$ is
homogeneous, $\R m$ is free of rank one, and
$s\cdot m=\lambda(s)m$ for every $s\in\Sset_{\deg m}$.  Its \emph{type} is $(\H(\deg m)\K,\lambda)$.

\draft{\par\smallskip\noindent\textbf{Types, weights, and the induced object.}
The coproduct defining $\Tset_{\H,\K}(\G)$ above is a disjoint union.  Thus the double
coset is part of a type even if two stabilizers, or two sets of character
extensions, happen to coincide.

For $m\in\mM^g$, the equality
$m=\sum_\lambda\bee_\lambda m$ is the decomposition of the
$\R\Sset_g$-module $\mM^g$ by its complete orthogonal idempotents.  If $m$
has weight $\lambda$, then
\[
 \bee_\lambda m
 =|\Sset_g|^{-1}\sum_s\lambda(s^{-1})\lambda(s)m=m.
\]
Conversely $s\bee_\lambda=\lambda(s)\bee_\lambda$ proves the weight
condition.

In $\mA_\theta(\mP)$ the balancing relation is
\[
 qs\otimes p=q\otimes s p=q\otimes\lambda(s)p\quad(s\in\Sset_C).
\]
It respects degrees because
$(qs)\cdot g_C=q\cdot(s\cdot g_C)=q\cdot g_C$.  Therefore the stated degree
does not depend on the chosen representative of a tensor.  Notice also that
no basis, finite generation, or projectivity of $\mP$ is assumed.\par\smallskip}

\begin{lemma}
\label{lem:orbit-induction}
Let $\mM$ be a $\G$-graded $(\RH,\RK)$-bimodule, let $C=\H g \K$, and put $\mM_C=\bigoplus_{x\in C}\mM^x$.  Then $\mM_C$ is a graded sub-bimodule and the map
$\Phi_g:\RQ\otimes_{\R\Sset_g}\mM^g\to\mM_C$ defined by
$\Phi_g(q\otimes m)=q\cdot m$ is a graded bimodule isomorphism.
\end{lemma}

\begin{proof}
Since $C$ is a $\Q$-orbit, $\mM_C$ is a graded sub-bimodule.  The map $\Phi_g$ is well-defined because, for
$s\in\Sset_g$,
\[
 \Phi_g(qs\otimes m)=qs\cdot m=q\cdot(s\cdot m)
 =\Phi_g(q\otimes s\cdot m).
\]
Choose representatives $T$ for $\Q/\Sset_g$.  Then
$\RQ=\bigoplus_{t\in T}t\R\Sset_g$, and therefore
\[
 \RQ\otimes_{\R\Sset_g}\mM^g
 \cong\bigoplus_{t\in T}t\otimes\mM^g.
\]
For each $t$, multiplication by the unit $t$ gives an $\R$-module
isomorphism $\mM^g\to\mM^{t\cdot g}$.  Moreover,
$t\Sset_g=t'\Sset_g$ if and only if $t\cdot g=t'\cdot g$.
Thus $\Phi_g$ is the direct sum of these isomorphisms and is graded.
\end{proof}

\draft{\par\smallskip\noindent\textbf{\Cref{lem:orbit-induction}, expanded.}
The sum $\mM_C$ is stable because acting by $q\in\Q$ sends a degree in the
orbit $C$ to another degree in the same orbit.  The formula for $\Phi_g$ is
balanced by the displayed computation in the proof, and it is $\RQ$-linear
because $\Phi_g(q'q\otimes m)=q'q\cdot m=q'\cdot\Phi_g(q\otimes m)$.

Choose one representative $t$ of each right coset of $\Sset_g$.  Every group
element has a unique form $ts$, so the group-element basis of $\RQ$ gives
$\RQ=\bigoplus_t t\R\Sset_g$.  Tensoring gives the displayed direct sum
without any assumption on $\mM^g$.  Multiplication by $t$ maps $\mM^g$ to
$\mM^{t\cdot g}$ and has inverse multiplication by $t^{-1}$.

Finally,
\[
 t\cdot g=t'\cdot g
 \Longleftrightarrow (t'^{-1}t)\cdot g=g
 \Longleftrightarrow t'^{-1}t\in\Sset_g
 \Longleftrightarrow t\Sset_g=t'\Sset_g.
\]
Hence distinct source summands map to distinct homogeneous components, every
component in $C$ occurs, and the restriction on each summand is bijective.
This proves injectivity, surjectivity, and preservation of degrees at once.\par\smallskip}

\begin{theorem}
\label{thm:classification}
Let $\mM$ be a $\G$-graded $(\RH,\RK)$-bimodule.  For
$\theta=(C,\lambda)\in\Tset_{\H,\K}(\G)$ put
$\mP_\theta(\mM)=\bee_\lambda \mM^{g_C}$.  Then
\begin{equation}
 \mM\cong_{\gr}
 \bigoplus_{\theta\in\Tset_{\H,\K}(\G)}
 \mA_\theta(\mP_\theta(\mM)).
 \label{eq:classification-decomposition}
\end{equation}
Conversely, any family $(\mP_\theta)_\theta$ of $\R$-modules defines a graded
bimodule by the direct sum in \eqref{eq:classification-decomposition}.

For $\R$-modules $\mP,\mP'$ and types $\theta,\phi$,
\begin{equation}
 \Hom^{\gr}_{(\RH,\RK)}
 \bigl(\mA_\theta(\mP),\mA_\phi(\mP')\bigr)
 \cong
 \begin{cases}
 \Hom_\R(\mP,\mP'),&\theta=\phi,\\
 0,&\theta\neq\phi.
 \end{cases}
 \label{eq:Hom-induced}
\end{equation}
Hence
\begin{equation}
 \mM\cong_{\gr}\mN
 \quad\Longleftrightarrow\quad
 \mP_\theta(\mM)\cong_\R\mP_\theta(\mN)
 \text{ for every }\theta\in\Tset_{\H,\K}(\G).
 \label{eq:isomorphism-criterion}
\end{equation}
\end{theorem}

\begin{proof}
By \cref{lem:orbit-induction}, $\mM=\bigoplus_C\mM_C$, and, for every
$C$, one has $\mM_C\cong_{\gr}\RQ\otimes_{\R\Sset_C}\mM^{g_C}$.
Since $\R\Sset_C\cong\R^{|\Sset_C|}$, we have
$\mM^{g_C}=\bigoplus_{\lambda\in\widehat{\Sset_C}}
 \bee_\lambda\mM^{g_C}$.
Tensoring this direct sum with $\RQ$ over $\R\Sset_C$ gives
\eqref{eq:classification-decomposition}.  Conversely, each summand
$\mA_\theta(\mP_\theta)$ is graded by
\eqref{eq:induced-bimodule-grading}, so any family $(\mP_\theta)_\theta$
gives a graded bimodule by direct sum.

If two types have different double cosets, every graded homomorphism between
the corresponding induced bimodules is zero because their supports are
disjoint.  Fix $C$.  The degree-$g_C$ component of
$\mA_{(C,\lambda)}(\mP)$ is $1\otimes\mP\cong\mP_\lambda$: among the cosets
of $\Sset_C$ in $\Q$, only $\Sset_C$ itself fixes $g_C$.  The induced module
is generated by this component under $\Q$, so a graded homomorphism is
uniquely determined by its restriction to degree $g_C$.  Restriction therefore
gives
\[
 \Homgr\bigl(\mA_{(C,\lambda)}(\mP),
             \mA_{(C,\mu)}(\mP')\bigr)
 \longrightarrow \Hom_{\R\Sset_C}(\mP_\lambda,\mP'_\mu).
\]
It is an isomorphism: an $\R\Sset_C$-map $f$ extends uniquely by
$q\otimes p\mapsto q\otimes f(p)$.  If $\lambda\neq\mu$, then
$\bee_\lambda$ acts as the identity on $\mP_\lambda$ and as zero on
$\mP'_\mu$.  Hence every $\R\Sset_C$-map is zero.  If
$\lambda=\mu$, the action of $\R\Sset_C$ on both modules is scalar through
$\lambda$, so every $\R$-linear map is $\R\Sset_C$-linear.  This proves
\eqref{eq:Hom-induced}.  Finally, a graded isomorphism $\mM\to\mN$ restricts to isomorphisms
$\mP_\theta(\mM)\to\mP_\theta(\mN)$ for every $\theta$.
Conversely, any family of such $\R$-module isomorphisms induces, by the
preceding Hom calculation, isomorphisms of the corresponding summands in
\eqref{eq:classification-decomposition}.  Their direct sum is a graded
isomorphism $\mM\to\mN$.  This proves
\eqref{eq:isomorphism-criterion}.
\end{proof}

\draft{\par\medskip\noindent\textbf{\Cref{thm:classification}, verification ledger.}
\begin{enumerate}[label=\textup{(\arabic*)}]
\item The degree set is the disjoint union of its $\Q$-orbits, so
$\mM=\bigoplus_C\mM_C$.  \Cref{lem:orbit-induction} reconstructs each
orbit summand from the single degree $g_C$.

\item Since $1=\sum_\lambda\bee_\lambda$ and the idempotents are orthogonal,
\[
 \mM^{g_C}=\bigoplus_\lambda\bee_\lambda\mM^{g_C}.
\]
Tensor product distributes over this finite direct sum.  On the
$\lambda$-summand the stabilizer acts by the scalar character $\lambda$, so
the induced summand is exactly $\mA_{(C,\lambda)}
(\mP_{(C,\lambda)}(\mM))$.

\item Conversely, the direct sum of an arbitrary family is legitimate:
every element of a module direct sum has finite support, and the action of a
fixed element of $\RQ$ acts within each summand.  The grading is the direct
sum of the gradings already checked on the induced objects.

\item If $C\neq D$, the source is supported on $C$ and the target on $D$.
A strict graded map sends the degree-$x$ source component into the degree-$x$
target component, which is zero for every $x\in C$.  Hence the map is zero.

\item Fix $C$.  A tensor has degree $g_C$ exactly when its coset
$q\Sset_C$ fixes $g_C$, hence exactly when $q\in\Sset_C$.  Balancing then
reduces it to $1\otimes p$.  Acting by all $q\in\Q$ on these tensors
generates the whole induced module.

\item An $\R\Sset_C$-map $f$ extends by
$q\otimes p\mapsto q\otimes f(p)$.  It is balanced because
$f(sp)=sf(p)$.  Conversely, any graded module map is determined by degree
$g_C$, so restriction and extension are inverse.

\item If $\lambda\neq\mu$ and $f:\mP_\lambda\to\mP'_\mu$ is a stabilizer
map, then
$f(p)=f(\bee_\lambda p)=\bee_\lambda f(p)=0$, since orthogonality makes
$\bee_\lambda$ zero on weight $\mu$.  If $\lambda=\mu$, both actions are
the same scalar action, so every $\R$-linear map commutes with it.

\item A graded isomorphism preserves every $C$ and every weight idempotent,
so it restricts to coefficient-module isomorphisms.  Conversely, inducing a
family of coefficient isomorphisms and taking their direct sum gives a strict
graded isomorphism.  This proves both directions of
\eqref{eq:isomorphism-criterion}, including zero coefficient modules and
infinite families.
\end{enumerate}\par\medskip}

\begin{remark}
\label{rem:fixed-type}
For each type $\theta$, the assignment $\mP\mapsto\mA_\theta(\mP)$ induces,
by \eqref{eq:Hom-induced}, a bijection between $\R$-module homomorphisms
$\mP\to\mP'$ and graded bimodule homomorphisms
$\mA_\theta(\mP)\to\mA_\theta(\mP')$.  Thus, already for one fixed type, the
graded isomorphism problem contains the isomorphism problem for arbitrary
$\R$-modules.  Ranks determine the answer only under additional hypotheses on
those modules.
\end{remark}

\begin{corollary}
\label{cor:finite-generation}
For a $\G$-graded $(\RH,\RK)$-bimodule $\mM$, the following are equivalent:
\begin{enumerate}[label=\textup{(\roman*)}]
 \item $\mM$ is finitely generated as an $(\RH,\RK)$-bimodule.
 \item $\mM$ is finitely generated as an $\RQ$-module.
 \item $\mM$ is finitely generated as an $\R$-module.
 \item only finitely many $\mP_\theta(\mM)$ are nonzero and every nonzero
       $\mP_\theta(\mM)$ is finitely generated over $\R$.
\end{enumerate}
\end{corollary}

\begin{proof}
The first two conditions are equivalent because the bimodule action is the
module action of the enveloping algebra
$\RH\otimes_\R(\RK)^{\mathrm{op}}\cong\RQ$, via
$(a\otimes b^{\mathrm{op}})m=amb$.  If $\mM$ is finitely generated
over $\RQ$, then it is finitely generated over $\R$ because $\RQ$ is finite
free over $\R$.  Conversely, an $\R$-generating set also generates $\mM$ over
$\RQ$, since $\R\subseteq\RQ$.  Thus (ii) and (iii) are equivalent.

For $\theta=(C,\lambda)$,
$\mA_\theta(\mP)_\R\cong
 \mP^{\,[\Q:\Sset_C]}$,
because $\RQ$ is free over $\R\Sset_C$ on representatives of
$\Q/\Sset_C$.  Hence \eqref{eq:classification-decomposition} shows that (iv)
implies (iii).  Conversely, if $\mM$ is finitely generated over $\R$, replace a finite set of
generators by all of their homogeneous components.  This is still a finite
$\R$-generating set, so only finitely many degrees, hence finitely many double
cosets, occur.  Since each $\widehat{\Sset_C}$ is finite, only
finitely many types can occur.  Each $\mM^{g_C}$ is an $\R$-direct summand of $\mM$ and is therefore
finitely generated.  Applying the idempotent $\bee_\lambda$ shows that each
$\mP_\theta(\mM)$ is a direct summand of $\mM^{g_C}$ and hence finitely
generated.  Thus (iii)
implies (iv).
\end{proof}

\draft{\par\smallskip\noindent\textbf{\Cref{cor:finite-generation}, expanded.}
The enveloping algebra action is not merely analogous to the bimodule action:
the formulas identify them exactly, so a generating set in one sense is a
generating set in the other.  If $x_1,\ldots,x_r$ generate over $\RQ$, then
the finitely many elements $q_ix_j$, with $(q_i)$ an $\R$-basis of $\RQ$,
generate over $\R$.  Conversely, an $\R$-generating set generates over
$\RQ$ because every scalar $r\in\R$ is already in $\RQ$.

For the last equivalence, remember that every element of a direct sum has
only finitely many nonzero homogeneous components.  Replacing finitely many
generators by all their homogeneous components therefore still gives a
finite set.  Its degrees contain the entire support: an $\R$-linear
combination cannot create a new degree.  Each $\mM^{g_C}$ is the image of the
$\R$-linear projection onto one degree, so it is a direct summand of $\mM$.
Applying the idempotent $\bee_\lambda$ gives another direct summand.  Direct
summands of finitely generated modules are finitely generated, with no
Noetherian hypothesis needed.\par\smallskip}

\begin{corollary}
\label{cor:projective-free-atomic}
Let $\mM$ be finitely generated.  Then:
\begin{enumerate}[label=\textup{(\roman*)}]
 \item $\mM$ is projective over $\R$ if and only if every $\mP_\theta(\mM)$ is
       projective over $\R$.
 \item $\mM$ is a finite direct sum of atomic graded bimodules if and only if every
       $\mP_\theta(\mM)$ is finite free over $\R$.
\end{enumerate}
In the second case, if $\mP_\theta(\mM)\cong \R^{m_\theta}$, then
\begin{equation}
 \mM\cong_{\gr}\bigoplus_\theta
 \mA_\theta(\R)^{\,m_\theta},
 \label{eq:atomic-multiplicity-decomposition}
\end{equation}
and the multiplicities $m_\theta$ are uniquely determined.
\end{corollary}

\begin{proof}
By \cref{cor:finite-generation}, only finitely many
$\mP_\theta(\mM)$ are nonzero.  As $\R$-modules,
\begin{equation}
 \mM\cong_\R
 \bigoplus_{\theta=(C,\lambda)}
 \mP_\theta(\mM)^{\,[\Q:\Sset_C]}.
 \label{eq:R-module-decomposition}
\end{equation}
If every $\mP_\theta(\mM)$ is projective, then \eqref{eq:R-module-decomposition}
shows that $\mM$ is projective.  Conversely, each
$\mP_\theta(\mM)$ is an $\R$-direct summand of $\mM$: first project to
$\mM^{g_C}$ using the grading, then apply $\bee_\lambda$.  Since
$\Sset_C$ fixes $g_C$, the latter idempotent preserves $\mM^{g_C}$.
Hence projectivity of $\mM$ implies projectivity of every
$\mP_\theta(\mM)$.  This proves (i).

If $\mP_\theta(\mM)\cong\R^{m_\theta}$, then
\eqref{eq:classification-decomposition} gives
\eqref{eq:atomic-multiplicity-decomposition}, and each
$\mA_\theta(\R)$ is atomic.  Conversely, let $\RH m\RK$ be atomic of type $(C,\lambda)$.  After replacing
$m$ by a translate of degree $g_C$, its degree-$g_C$ component is $\R m$:
only elements of $\Sset_C$ preserve that degree, and they act on $m$ through
$\lambda$.  Hence its only nonzero coefficient module is a copy of $\R$.
It follows that a finite direct sum of atomic bimodules has finite free
$\mP_\theta(\mM)$ for every $\theta$.  \Cref{thm:classification} also
shows that the integers $m_\theta$ are unique.
\end{proof}

\draft{\par\smallskip\noindent\textbf{\Cref{cor:projective-free-atomic}, expanded.}
Equation~\eqref{eq:R-module-decomposition} repeats each coefficient module
once for each coset of $\Sset_C$.  A finite direct sum of projective modules
is projective.  In the reverse direction, the composite
\[
 \mM\longrightarrow\mM^{g_C}
 \xrightarrow{\,\bee_\lambda\,}\bee_\lambda\mM^{g_C}
\]
is an idempotent $\R$-linear projection, so the coefficient module is a
direct summand of the projective module $\mM$ and is projective.  Nothing here
says that a projective module is free.

If a coefficient module is $\R^{m_\theta}$, induction distributes over this
finite direct sum and gives $m_\theta$ copies of $\mA_\theta(\R)$.  Conversely,
for an atomic generator $m$ of degree $g_C$, only stabilizer elements preserve
that degree, and each acts by $\lambda$.  Hence the degree-$g_C$ component is
exactly $\R m$, so its coefficient module is one copy of $\R$.  Applying this
to every atomic summand gives finite free coefficient modules.  Uniqueness
comes from the strict graded classification: the coefficient module at each
fixed type is $\R^{m_\theta}$, whose finite rank is well-defined over the
nonzero commutative ring $\R$.\par\smallskip}

For $\theta=(C,\lambda)$ put $\mA_\theta=\mA_\theta(\R)$.  The degree-$g_C$
component is $1\otimes\R\cong\R$, so $1\otimes1$ is a homogeneous free
weight vector of degree $g_C$ and weight $\lambda$.  Hence $\mA_\theta$ is
atomic of type $\theta$.  Conversely, let $\mM=\RH m\RK$ be atomic of type $(C,\lambda)$.  Replacing
$m$ by $q\cdot m$ for a suitable $q\in\Q$, we may assume
$\deg m=g_C$.  If $s\in\Sset_C$, then
$s(q\cdot m)=q(s\cdot m)=\lambda(s)(q\cdot m)$ because $\Q$ is abelian.  Hence
the weight remains $\lambda$.  Every element of $\mM^{g_C}$ is an $\R$-linear combination of elements
$h m k$ with $h g_C k=g_C$, hence with $(h,k)\in\Sset_C$.  Each such element
is $\lambda(h,k)m$.  Therefore $\mM^{g_C}=\R m$.  Thus
\cref{lem:orbit-induction} gives
\[
 \mM\cong_{\gr}\RQ\otimes_{\R\Sset_C}\R m
 \cong_{\gr}\mA_{(C,\lambda)}.
\]
For the incidence-algebra bimodules in
\cite[Section~7]{HindiSchutzer2026Incidence}, each nonzero
$\mP_\theta(\mM)$ is isomorphic to $\R$, and the extension sets of the
occurring types are pairwise disjoint.  This is the multiplicity-free case
treated in \cref{sec:multiplicity-free}.

\draft{\par\smallskip\noindent\textbf{Final atomic identification.}
If the original atomic generator has degree $x\in C$, choose $q$ with
$q\cdot x=g_C$.  Since $q$ is a unit, $q\cdot m$ still generates the same
bimodule.  For $s\in\Sset_C=\Sset_x$, abelianness of $\Q$ gives
$s(qm)=q(sm)=\lambda(s)qm$, so the weight is unchanged.  Any element of
degree $g_C$ is a sum of terms $h(qm)k$ with $hg_Ck=g_C$.  The pair $(h,k)$
then lies in $\Sset_C$, so this term is the scalar $\lambda(h,k)qm$.
Therefore the entire degree-$g_C$ component is the single free line
$\R(qm)$, and orbit induction identifies the atomic bimodule with
$\mA_{(C,\lambda)}$.\par\smallskip}

\section{Yetter-Drinfel'd modules}
\label{sec:YD}

In this section $\R=\F$ is a field and the standing splitting assumptions
remain in force.  We compare \cref{thm:classification} with the induction construction of
Dos Santos and Yasumura \cite[Definition~3.1, Remark~3.2 and Proposition~3.4]{DSY2025}.
Their Proposition~3.4 is stated for finite $\G$.  Its balancing calculation
does not use this hypothesis, so the same formula defines the functor for
arbitrary $\G$.  Let
$A=\H\cap\K$ and embed $A$ in $\Q=\H\times\K^{\mathrm{op}}$ by
$\delta(a)=(a,a^{-1})$.  For every subgroup $B\leq A$ we write
$\Delta B=\delta(B)=\{(b,b^{-1}):b\in B\}$.  If $L\leq\Q$ and $\lambda\in\widehat L$, then $\F_\lambda$ denotes the
one-dimensional $\F L$-module on which $l$ acts as multiplication by
$\lambda(l)$.  For a cardinal $m$, $\F_\lambda^m$ means the direct sum of $m$
copies of $\F_\lambda$.  For subgroups $D\leq S\leq\Q$ and an
$\F D$-module $\mW$, we use $\Ind_D^S\mW=\F S\otimes_{\F D}\mW$.

Following \cite[Definition~3.1 and Remark~3.2]{DSY2025}, an
$A$-Yetter-Drinfel'd $\G$-graded module is a $\G$-graded left $\F A$-module
$\mV=\bigoplus_{g\in \G}\mV^g$ satisfying
$a\mV^g\subseteq \mV^{aga^{-1}}$ for $a\in A$ and $g\in\G$.
The functor considered in \cite[Proposition~3.4]{DSY2025} is
\begin{equation}
 \mV\longmapsto
 \mV^\uparrow=\FQ\otimes_{\F A}\mV,
 \label{eq:YD-up-functor}
\end{equation}
where the right $\F A$-module structure on $\FQ$ is
$q\cdot a=q\delta(a)$, and $\deg(q\otimes v)=q\cdot\deg(v)$.  This grading is compatible with the balancing relation: if $v$ is homogeneous,
then
\[
 \deg(q\delta(a)\otimes v)
 =q\cdot(a(\deg v)a^{-1})
 =\deg(q\otimes av).
\]

For $g\in\G$, put $C_A(g)=\{a\in A:ag=ga\}$ and
$D_g=\Delta C_A(g)=\Sset_g\cap\Delta A$.

\draft{\par\smallskip\noindent\textbf{Checking the Yetter-Drinfel'd setup.}
The embedding $\delta$ uses the inverse because the second factor is
$\K^{\mathrm{op}}$.  For $a,b\in A$,
\[
 \delta(a)\delta(b)
 =(ab,(b^{-1}a^{-1})^{\mathrm{op}})
 =(ab,((ab)^{-1})^{\mathrm{op}})=\delta(ab).
\]
If $v$ has degree $x$, then $av$ has degree $axa^{-1}$.  The tensor relation
$q\delta(a)\otimes v=q\otimes av$ is homogeneous because both sides have
degree $q\cdot(axa^{-1})$.  Thus the induced grading is well-defined.
Finally,
\[
 \delta(a)\in\Sset_g
 \Longleftrightarrow aga^{-1}=g
 \Longleftrightarrow ag=ga,
\]
so $\Sset_g\cap\Delta A=\Delta C_A(g)$, including for a trivial
centralizer.\par\smallskip}

\begin{lemma}
\label{lem:YD-orbit}
Let $\mV$ be an $A$-Yetter-Drinfel'd module supported on the $A$-conjugacy
orbit of $g$, and put $\mW=\mV^g$.  Then, as a graded $\F A$-module,
\[
 \mV\cong \F A\otimes_{\F C_A(g)}\mW,
\]
where $\mW$ carries the restricted $C_A(g)$-action.  Hence
\[
 \mV^\uparrow
 \cong_{\gr}
 \FQ\otimes_{\F D_g}\mW.
\]
The degree-$g$ component of $\mV^\uparrow$, viewed as an $\F\Sset_g$-module, is
\[
 \Ind_{D_g}^{\Sset_g}\mW.
\]
\end{lemma}

\begin{proof}
For $a\in A$, multiplication by $a$ maps $\mV^g$ isomorphically onto
$\mV^{aga^{-1}}$: the defining inclusion gives one direction and $a^{-1}$
gives the reverse inclusion.  Choose representatives $T$ for $A/C_A(g)$.  Then
$\mV=\bigoplus_{a\in T}a\mW$, and the map
$\F A\otimes_{\F C_A(g)}\mW\to\mV$ defined by $a\otimes w\mapsto aw$
is the direct sum of the isomorphisms
$\mW\to\mV^{aga^{-1}}$.  Hence it is a graded isomorphism, with
$\deg(a\otimes w)=aga^{-1}$.  Compare \cite[Lemma~3.3]{DSY2025}.  Via $C_A(g)\xrightarrow{\delta}D_g$, the tensor-product isomorphism is
explicitly
\[
 q\otimes(a\otimes w)\longmapsto q\delta(a)\otimes w,
\]
and gives
\[
 \mV^\uparrow
 \cong_{\gr}
 \FQ\otimes_{\F D_g}\mW.
\]
Choose representatives of $\Q/D_g$.  The tensors belonging to a representative
$qD_g$ have degree $q\cdot g$, and $q\cdot g=g$ precisely when
$q\in\Sset_g$.  The representatives of degree $g$ are therefore those of
$\Sset_g/D_g$, so the degree-$g$ component is
\[
 \F\Sset_g\otimes_{\F D_g}\mW
 =\Ind_{D_g}^{\Sset_g}\mW,
\]
with the natural $\F\Sset_g$-action.
\end{proof}

\draft{\par\smallskip\noindent\textbf{\Cref{lem:YD-orbit}, expanded.}
For $a\in A$, the defining condition gives
$a\mV^g\subseteq\mV^{aga^{-1}}$.  Applying it to $a^{-1}$ supplies the
inverse inclusion, so multiplication by $a$ is an isomorphism of the two
components.  Moreover $aga^{-1}=bgb^{-1}$ exactly when
$b^{-1}a\in C_A(g)$, so the distinct orbit degrees are indexed by
$A/C_A(g)$.

The map $a\otimes w\mapsto aw$ is balanced because
$ac\otimes w=a\otimes cw$ for $c\in C_A(g)$, and both sides map to $acw$.
On coset representatives it is the direct sum of the componentwise
isomorphisms just checked.

For the tensor-associativity map, the inner relation is respected because
$c$ acts on $\mW$ as $\delta(c)$, and the outer relation
$q\delta(a')\otimes(a\otimes w)=q\otimes(a'a\otimes w)$ has the same image
$q\delta(a'a)\otimes w$ on both sides.  An inverse is represented by
$q\otimes w\mapsto q\otimes(1\otimes w)$.

Finally, the coset summand represented by $qD_g$ has degree $q\cdot g$.
It has degree $g$ exactly when $q\in\Sset_g$.  Modulo $D_g$, these are the
cosets of $D_g$ in $\Sset_g$, giving
$\F\Sset_g\otimes_{\F D_g}\mW$.\par\smallskip}

By \cref{lem:orbit-induction}, the assignments $\mM\mapsto\mM^g$ and
$\mP\mapsto\FQ\otimes_{\F\Sset_g}\mP$ are inverse up to graded isomorphism between graded $(\FH,\FK)$-bimodules
supported on $C=\H g\K$ and $\F\Sset_g$-modules.  By the \emph{essential
image} of a functor we mean the objects isomorphic to objects in its image.

\begin{theorem}
\label{thm:YD-essential-image}
Fix a double coset $C=\H g \K$.  Under the equivalence between graded
$(\FH,\FK)$-bimodules supported on $C$ and $\F\Sset_g$-modules, the objects in the
essential image of the functor \eqref{eq:YD-up-functor} are precisely the
$\F\Sset_g$-modules in the essential image of the induction functor
$\Ind_{D_g}^{\Sset_g}$
from $\F D_g$-modules to $\F\Sset_g$-modules.

Write
$\mP=\bigoplus_{\lambda\in\widehat{\Sset_g}}\F_\lambda^{m_\lambda}$,
where $m_\lambda=\dim_\F(\bee_\lambda\mP)$ may be an infinite cardinal.
Then $\mP$ occurs if and only if
\begin{equation}
 m_\lambda=m_\mu\quad\text{whenever }\lambda|_{D_g}=\mu|_{D_g}.
 \label{eq:essential-image-multiplicities}
\end{equation}
The functor is essentially surjective on graded bimodules supported on $C$
if and only if
\begin{equation}
 \H\cap g\K g^{-1}=C_{\H\cap \K}(g).
 \label{eq:YD-local-surjectivity-group}
\end{equation}
It is essentially surjective on all graded $(\FH,\FK)$-bimodules if and only if
\eqref{eq:YD-local-surjectivity-group} holds for every double coset $\H g \K$.
\end{theorem}

\begin{proof}
First restrict a Yetter-Drinfel'd module to the degrees in $C$.  This is an
$A$-submodule because
$a t a^{-1}=\delta(a)\cdot t\in C$ for $a\in A$ and $t\in C$.
Decompose it into $A$-conjugacy orbits.  If $t=hgk\in C$ and $a\in A$, then,
since $A\leq\H\cap\K$ and $\H,\K$ are abelian,
\[
 at=ta
 \iff ahgk=hgka
 \iff ag=ga.
\]
Thus $C_A(t)=C_A(g)$ and $D_t=D_g$ for every $t\in C$.  For an
$A$-conjugacy orbit choose $t$ in that orbit and put $\mW_t=\mV^t$.
By \cref{lem:YD-orbit}, its image is
$\FQ\otimes_{\F D_g}\mW_t$.  Since $t$ and $g$ lie in the same $\Q$-orbit, choose $q_0\in\Q$ with
$q_0\cdot t=g$.  The degree-$g$ component is then
\[
 q_0\F\Sset_g\otimes_{\F D_g}\mW_t
 \cong \F\Sset_g\otimes_{\F D_g}\mW_t,
\]
as an $\F\Sset_g$-module, via $q_0s\otimes w\mapsto s\otimes w$.  This is
$\F\Sset_g$-linear because $\Q$ is abelian.  Summing over the $A$-conjugacy
orbits gives $\Ind_{D_g}^{\Sset_g}\mW$, where
$\mW=\bigoplus_t\mW_t$.

Conversely, let $\mW$ be any $\F D_g$-module.  Via $\delta$ regard it as an
$\F C_A(g)$-module and choose representatives $T$ for $A/C_A(g)$.  Then
\[
 \mV=\F A\otimes_{\F C_A(g)}\mW
     =\bigoplus_{a\in T}a\otimes\mW,
\]
with $\deg(a\otimes w)=aga^{-1}$.
Two representatives give the same conjugate of $g$ exactly when they lie in
the same coset of $C_A(g)$, so this is a grading.  Left multiplication by
$b\in A$ sends the component of degree $aga^{-1}$ to that of degree
$baga^{-1}b^{-1}$.  Thus $\mV$ is an $A$-Yetter-Drinfel'd
module.  Applying \eqref{eq:YD-up-functor} and \cref{lem:YD-orbit} gives
degree-$g$ module $\Ind_{D_g}^{\Sset_g}\mW$.  This proves the first
assertion.

Since $\F$ splits the finite abelian groups $D_g\leq\Sset_g$, every
$\F D_g$-module is the direct sum of the images of its Fourier idempotents,
and restriction $\widehat{\Sset_g}\to\widehat{D_g}$ is surjective.  If
$\eta\in\widehat{D_g}$, view $\bee_\eta\in\F D_g$ inside
$\F\Sset_g$.  Then
\[
 \bee_\eta
 =\sum_{\substack{\lambda\in\widehat{\Sset_g}\\
                    \lambda|_{D_g}=\eta}}\bee_\lambda.
\]
Indeed, applying any character of $\Sset_g$ gives the same value on both sides,
and the character table is invertible by \eqref{eq:character-orthogonality}.
Since $\F_\eta\cong\F D_g\bee_\eta$, it follows that
\[
 \Ind_{D_g}^{\Sset_g}\F_\eta
 \cong \F\Sset_g\bee_\eta
 =\bigoplus_{\substack{\lambda\in\widehat{\Sset_g}\\
                       \lambda|_{D_g}=\eta}}\F\bee_\lambda
 \cong
 \bigoplus_{\substack{\lambda\in\widehat{\Sset_g}\\
                       \lambda|_{D_g}=\eta}}\F_\lambda.
\]
Hence an induced module has multiplicity constant on each fiber of restriction.
Conversely, if the multiplicities $m_\lambda$ satisfy
\eqref{eq:essential-image-multiplicities}, choose
$n_\eta=m_\lambda$ for any $\lambda$ restricting to $\eta$ and put
$\mW=\bigoplus_\eta\F_\eta^{n_\eta}$.  Then
$\Ind_{D_g}^{\Sset_g}\mW\cong\mP$.  This proves
\eqref{eq:essential-image-multiplicities}.

The first projection identifies $\Sset_g$ with
$\H\cap g\K g^{-1}$ and maps $D_g$ onto $C_{\H\cap\K}(g)$.  Since the
projection is injective, $D_g=\Sset_g$ is equivalent to
\eqref{eq:YD-local-surjectivity-group}.  If $D_g=\Sset_g$, induction is the
identity.  If $D_g<\Sset_g$, every fiber of
$\widehat{\Sset_g}\to\widehat{D_g}$ has $[\Sset_g:D_g]>1$ elements, so no
one-dimensional $\F\Sset_g$-module belongs to the image.

Every graded bimodule is the direct sum of its sub-bimodules supported
on the double cosets, and the functor \eqref{eq:YD-up-functor} preserves this
decomposition.  If the local condition holds for every $C$, choose a
Yetter-Drinfel'd preimage for each double-coset summand and take their direct
sum.  Tensor product commutes with direct sums.  Conversely, if the condition
fails for some $C$, the orbit-induced graded bimodule associated with a
one-dimensional $\F\Sset_g$-module on that coset is not in the global
essential image.  Hence global essential surjectivity holds exactly when the
condition holds for every $C$.
\end{proof}

\draft{\par\medskip\noindent\textbf{\Cref{thm:YD-essential-image}, detailed verification.}
\begin{enumerate}[label=\textup{(\arabic*)}]
\item Degrees in $C$ form an $A$-stable set because conjugation by $a$ is
the action of $\delta(a)\in\Q$.  Hence restriction to $C$ and decomposition
into $A$-conjugacy orbits are legitimate module decompositions.

\item If $t=hgk$, then
\[
 at=ta\Longleftrightarrow ahgk=hgka
 \Longleftrightarrow hagk=hgak
 \Longleftrightarrow ag=ga.
\]
The rearrangements use $a,h\in\H$ and $a,k\in\K$.  Thus $C_A(t)$ and
$D_t$ are constant on the double coset.

\item An orbit representative $t$ need not be $g$.  Choose $q_0$ with
$q_0\cdot t=g$.  Degree-$g$ tensors are represented by $q_0s$ with
$s\in\Sset_g$.  The map $q_0s\otimes w\mapsto s\otimes w$ is balanced over
$D_g$ and is $\F\Sset_g$-linear: for $r\in\Sset_g$,
$rq_0s=q_0rs$ because $\Q$ is abelian.

\item Conversely, replacing $a\otimes w$ by
$ac\otimes w=a\otimes cw$ does not change $aga^{-1}$ because
$c\in C_A(g)$.  Multiplication by $b\in A$ changes this degree to
$baga^{-1}b^{-1}$, proving the Yetter-Drinfel'd condition.

\item Apply any character $\gamma$ of $\Sset_g$ to
\[
 \bee_\eta=\sum_{\lambda|_{D_g}=\eta}\bee_\lambda.
\]
Both sides evaluate to $1$ exactly when $\gamma|_{D_g}=\eta$, and to $0$
otherwise.  Since the character table is invertible, the elements are equal.
Thus inducing $\F_\eta$ produces once each character above $\eta$.

\item Direct sums now show that every induced module has multiplicity
constant on each restriction fiber.  Conversely, constant multiplicities
define $\mW=\bigoplus_\eta\F_\eta^{n_\eta}$ whose induction is $\mP$.
This remains valid for infinite cardinals because tensor product commutes with
arbitrary direct sums.

\item The first projection maps $\Sset_g$ isomorphically onto
$\H\cap g\K g^{-1}$ and maps $D_g$ onto
$C_{\H\cap\K}(g)$.  Hence the two equality criteria are equivalent.  If
$D_g<\Sset_g$, a single character module cannot be induced because its
restriction fiber has more than one character and induced multiplicities must
be equal throughout that fiber.

\item For the global converse, choose a preimage for every double-coset
summand and take their direct sum.  Tensoring with $\FQ$ returns the direct
sum of the local images.  This is the required global object.  Local
surjectivity alone would not finish the argument without this assembly.
\end{enumerate}\par\medskip}

\begin{corollary}
\label{cor:abelian-YD}
If $\G$ is abelian, then $D_g=\Sset_g=\Delta(\H\cap \K)$ for every $g\in \G$.
Hence \cref{thm:YD-essential-image} recovers the essential-surjectivity
assertion of \cite[Corollary~3.8]{DSY2025} for the functor
\eqref{eq:YD-up-functor} in the abelian case.
\end{corollary}

\begin{proof}
If $\G$ is abelian, $hgk=g$ if and only if
$h=k^{-1}\in\H\cap\K$, so $\Sset_g=D_g=\Delta(\H\cap\K)$.
\end{proof}

\draft{\par\smallskip\noindent\textbf{\Cref{cor:abelian-YD}.}
When $\G$ is abelian, $hgk=g$ is equivalent to $hk=1$, hence
$h=k^{-1}$.  Such a pair belongs to $\H\times\K$ exactly when
$h\in\H\cap\K$, so every stabilizer element is
$\delta(h)=(h,h^{-1})$.  Also every element of $A=\H\cap\K$ commutes with
$g$, so $C_A(g)=A$.  Therefore both $\Sset_g$ and $D_g$ are the same
diagonal subgroup.\par\smallskip}

\begin{example}
\label{ex:A4-counterexample}
Let $\G=A_4$ and $\H=\K=V_4$, where $V_4$ is the normal Klein four subgroup of double transpositions, and let
$g=(123)$.  Since $V_4\triangleleft A_4$,
\[
 \Sset_g=\{(gk^{-1}g^{-1},k):k\in V_4\}\cong V_4.
\]
Conjugation by $g$ cyclically permutes the three nonidentity elements of
$V_4$, so $C_{V_4}(g)=1$ and $D_g=1$.  Thus $D_g<\Sset_g$, so \eqref{eq:YD-up-functor} is not essentially surjective.
For $\lambda\in\widehat{\Sset_g}$, the atomic bimodule
\[
 \mA_{(V_4 g V_4,\lambda)}(\F)
\]
has degree-$g$ weight space $\F_\lambda$, whereas induction from $D_g=1$ gives equal
multiplicity to all four characters of $\Sset_g$.  This answers negatively the
question following \cite[Corollary~3.8]{DSY2025}.
\end{example}

\draft{\par\smallskip\noindent\textbf{The $A_4$ calculation.}
Write
\[
 V_4=\{1,(12)(34),(14)(23),(13)(24)\}.
\]
Conjugation by $g=(123)$ cycles the nonidentity elements:
\[
 (12)(34)\mapsto(14)(23)\mapsto(13)(24)\mapsto(12)(34).
\]
Hence $C_{V_4}(g)=1$.  Normality ensures that
$h=gk^{-1}g^{-1}$ lies in $V_4$ for every $k\in V_4$, and then $hgk=g$.
Conversely the stabilizer equation determines $h$ from $k$, so
$\Sset_g\cong V_4$ and has four elements, whereas
$D_g=\Delta C_{V_4}(g)=1$.

Induction from $1$ sends a vector space $W$ to
$\F\Sset_g\otimes W$.  The regular module of the split group $V_4$
contains each of its four characters once, so every character occurs with
multiplicity $\dim W$.  The atomic module in the example has multiplicity
one at one character and zero at the other three.  It therefore violates the
fiber condition.\par\smallskip}

\begin{remark}
\label{rem:SSY-free-product}
In \cite[Remark~24]{SSY2020}, $\G=\H*C_2$, $C_2=\langle w\rangle$, and
\cite[Corollary~23]{SSY2020} fails for the graded $(\FH,\FH)$-bimodule
$\FH\otimes_\F \FH$ with $\deg(h_1\otimes h_2)=h_1wh_2$.  By the normal
form theorem for free products,
\[
 \H\cap wHw^{-1}=1,
\]
so $\Sset_w(\H,\H)=1$.  The map $\FQ\to\FH\otimes_\F\FH$ defined by
$(h_1,h_2)\mapsto h_1\otimes h_2$ is then a graded $(\FH,\FH)$-bimodule isomorphism from
$\mA_{(\H w\H,1)}(\F)$ onto their example.  This does not contradict the
last assertion of \cite[Remark~24]{SSY2020}: viewed as a left $\FQ$-module,
our grading is by the $\Q$-set $\G$, not a $\G$-graded module structure over a
$\G$-graded algebra $\FQ$.  Thus the stabilizer in this nonabelian example is
trivial, rather than $\H\cap\H=\H$.
\end{remark}

\draft{\par\smallskip\noindent\textbf{The free-product example.}
If $x\in\H\cap w\H w^{-1}$, write $x=whw^{-1}=whw$.  When $h\neq1$, the
right side is a reduced free-product word of length three and cannot be in
the free factor $\H$.  Thus $h=1$ and $x=1$.

Since $\Sset_w=1$, the atomic induced module is $\FQ$.  The displayed map
sends its group-element basis bijectively to the tensor-product basis.  Left
multiplication changes the first tensor factor and right multiplication
changes the second.  The opposite-group multiplication gives the correct order
on the latter.  The basis element $(h_1,h_2)$ has degree $h_1wh_2$, so the map
is also strict graded.\par\smallskip}

\section{Peirce decomposition}
\label{sec:Fourier-Peirce}

We compare \cref{thm:classification} with the Peirce decomposition
of $\mM$.  Since $\Q=\H\times\K^{\mathrm{op}}$ is abelian, identify
$\hQ\cong\hH\times\hK$ and write $\xi=(\chi,\rho)$, where
\[
 \xi(h,k^{\mathrm{op}})=\chi(h)\rho(k).
\]
For such a pair we write
\[
 \bee_\xi=\bee_\chi\otimes \bee_\rho\in \RQ.
\]
Its action on an $(\RH,\RK)$-bimodule is the projection
$m\mapsto \bee_\chi m \bee_\rho$.  Since the Fourier idempotents form complete
sets of orthogonal idempotents in $\RH$ and $\RK$, every bimodule has the
(ungraded) Peirce decomposition
\[
 \mM=\bigoplus_{(\chi,\rho)\in\hH\times\hK}\bee_\chi\mM\bee_\rho.
\]
We call $\bee_\chi\mM\bee_\rho$ the $(\chi,\rho)$-\emph{Peirce component} of
$\mM$.  If $\xi=(\chi,\rho)$ and $m$ belongs to this component, then, for
$q=(h,k^{\mathrm{op}})$,
\[
 q\cdot m=h\bee_\chi m\bee_\rho k
 =\chi(h)\rho(k)m=\xi(q)m.
\]
Hence an $\R$-linear map between two $\xi$-Peirce components is automatically
an $\RQ$-module map.

For $S\leq\Q$ and $\lambda\in\widehat S$, let
$E(S,\lambda)=\{\xi\in\hQ:\xi|_S=\lambda\}$.  Thus $E(S,\lambda)$ is the set of extensions of $\lambda$ to $\Q$.  By
\cref{lem:split-stabilizers} it has cardinality $[\Q:S]$.  For a type
$\theta=(C,\lambda)$ write $E_\theta=E(\Sset_C,\lambda)$.

\draft{\par\smallskip\noindent\textbf{Peirce components before the calculation.}
A character of the direct product $\Q$ is determined by its restrictions to
$\H$ and $\K$, giving $\hQ\cong\hH\times\hK$.  Under the enveloping-algebra
action, $\bee_\chi\otimes\bee_\rho$ sends $m$ to
$\bee_\chi m\bee_\rho$.  The opposite multiplication is already built into
$\RQ$.

Expanding the two identities into Fourier idempotents gives
\[
 m=\sum_{\chi,\rho}\bee_\chi m\bee_\rho,
\]
and orthogonality makes the sum direct.  On the $\xi$-component every
$q\in\Q$ acts by the scalar $\xi(q)$.  Therefore any $\R$-linear map between
matching $\xi$-components automatically commutes with the whole bimodule
action.\par\smallskip}

\begin{proposition}
\label{prop:Peirce-induced}
Let $\theta=(C,\lambda)$ and let $\mP$ be an $\R$-module.  For
$\xi=(\chi,\rho)\in\hQ$ one has
\[
 \bee_\chi \mA_\theta(\mP) \bee_\rho
 \cong_\R
 \begin{cases}
  \mP,&\xi\in E_\theta,\\
  0,&\xi\notin E_\theta.
 \end{cases}
\]
\end{proposition}

\begin{proof}
Put $S=\Sset_C$.  Since $\bee_\xi\RQ=\R\bee_\xi$ and
$r\mapsto r\bee_\xi$ identifies $\R$ with $\R\bee_\xi$,
\[
 \bee_\xi\mA_\theta(\mP)
 =\bee_\xi\RQ\otimes_{\R S}\mP_\lambda.
\]
Let $\bee_\lambda\in\R S$ be the Fourier idempotent attached to $\lambda$.
Because $\bee_\lambda$ acts as the identity on $\mP_\lambda$,
\[
 \bee_\xi\otimes p=\bee_\xi\bee_\lambda\otimes p.
\]
Using $s\bee_\xi=\xi(s)\bee_\xi$ for $s\in S$ and character
orthogonality in $S$,
\[
 \bee_\xi\bee_\lambda
 =\frac1{|S|}\sum_{s\in S}\lambda(s^{-1})\xi(s)\bee_\xi
 =\begin{cases}
   \bee_\xi,&\xi|_S=\lambda,\\
   0,&\xi|_S\neq\lambda.
  \end{cases}
\]
Hence the Peirce component is zero if $\xi|_S\neq\lambda$.  If
$\xi|_S=\lambda$, the map $\mP\to\bee_\xi\mA_\theta(\mP)$ defined by
$p\mapsto\bee_\xi\otimes p$ is an isomorphism.  Indeed, every tensor has this form, and the inverse
$r\bee_\xi\otimes p\mapsto rp$ is balanced because
$\bee_\xi s=\lambda(s)\bee_\xi$.
\end{proof}

\draft{\par\smallskip\noindent\textbf{\Cref{prop:Peirce-induced}, tensor check.}
The ideal $\bee_\xi\RQ$ is one coordinate copy of $\R$, so every element is
uniquely $r\bee_\xi$.  Because $\bee_\lambda p=p$, balancing gives
$\bee_\xi\otimes p=\bee_\xi\bee_\lambda\otimes p$.  Directly,
\[
 \bee_\xi\bee_\lambda
 =|S|^{-1}\sum_s\lambda(s)^{-1}\xi(s)\bee_\xi.
\]
Orthogonality makes the scalar $1$ when $\xi|_S=\lambda$ and $0$ otherwise.

In the nonzero case every tensor is $r\bee_\xi\otimes p$, so the map from
$\mP$ is surjective.  Its inverse respects balancing because
\[
 r\bee_\xi s\otimes p\longmapsto r\xi(s)p
 \quad\text{and}\quad
 r\bee_\xi\otimes sp\longmapsto r\lambda(s)p,
\]
and these agree.  The two maps compose to the identity, proving injectivity
as well.\par\smallskip}

\begin{theorem}
\label{thm:Fourier-Peirce-formula}
Let $\mM$ be a $\G$-graded $(\RH,\RK)$-bimodule.  Then for every
$\xi=(\chi,\rho)\in\hQ$,
\[
 \bee_\chi \mM \bee_\rho
 \cong_\R
 \bigoplus_{\substack{\theta\in\Tset_{\H,\K}(\G)\\
                       \xi\in E_\theta}}
 \mP_\theta(\mM).
\]
\end{theorem}

\begin{proof}
Apply $\bee_\xi$ to the decomposition
\eqref{eq:classification-decomposition} and use
\cref{prop:Peirce-induced} on every summand.
\end{proof}

\draft{\par\smallskip\noindent\textbf{\Cref{thm:Fourier-Peirce-formula}.}
Multiplication by $\bee_\xi$ commutes with a direct sum.  Applying it to the
classification therefore gives the direct sum of the $\xi$-components of all
induced summands.  \Cref{prop:Peirce-induced} identifies that
component with $\mP_\theta(\mM)$ exactly for $\xi\in E_\theta$, and with zero
otherwise.  Infinite type sets cause no problem because every element of a
module direct sum has finite support.\par\smallskip}

\begin{proposition}
\label{prop:extensions-intersection}
Let $S,T\leq \Q$, let $\lambda\in\widehat S$ and
$\mu\in\widehat T$, and put $U=S\cap T$.  Then
\[
 E(S,\lambda)\cap E(T,\mu)\neq\varnothing
 \quad\Longleftrightarrow\quad
 \lambda|_U=\mu|_U.
\]
If these conditions hold, there is a unique character
$\nu\in\widehat{ST}$ extending both $\lambda$ and $\mu$, and
\[
 E(S,\lambda)\cap E(T,\mu)=E(ST,\nu).
\]
Hence
\[
 |E(S,\lambda)\cap E(T,\mu)|
 =\begin{cases}
   [\Q:ST],&\lambda|_U=\mu|_U,\\
   0,&\text{otherwise.}
  \end{cases}
\]
\end{proposition}

\begin{proof}
A common extension of $\lambda$ and $\mu$ has equal restrictions to $U$.
Conversely, assume $\lambda|_U=\mu|_U$ and define
$\nu(st)=\lambda(s)\mu(t)$ for $s\in S$ and $t\in T$.
If $st=s't'$, then $s'^{-1}s=t't^{-1}\in U$, and equality of the restrictions
gives $\lambda(s)\mu(t)=\lambda(s')\mu(t')$.  Thus $\nu$ is well-defined.
Since $\Q$ is abelian, $\nu$ is a character of $ST$, and it is the unique
character extending both $\lambda$ and $\mu$.  By
\cref{lem:split-stabilizers}, $\nu$ extends to $\Q$.  The common extensions
of $\lambda$ and $\mu$ are therefore exactly $E(ST,\nu)$, which has
$[\Q:ST]$ elements.
\end{proof}

\draft{\par\smallskip\noindent\textbf{\Cref{prop:extensions-intersection}, expanded.}
A common extension gives equal restrictions on $U$.  Conversely,
$ST$ is a subgroup because $\Q$ is abelian.  If $st=s't'$, then
$u=s'^{-1}s=t't^{-1}\in U$, and
\[
 \lambda(s)\mu(t)
 =\lambda(s'u)\mu(u^{-1}t')
 =\lambda(s')\lambda(u)\mu(u)^{-1}\mu(t')
 =\lambda(s')\mu(t').
\]
This checks well-definedness of $\nu$.  Commutativity gives
\[
 \nu(sts't')=\lambda(ss')\mu(tt')=\nu(st)\nu(s't').
\]
Any common extension
must take $st$ to $\lambda(s)\mu(t)$, so $\nu$ is unique.
\Cref{lem:split-stabilizers} extends it to $\Q$ and counts the extensions.\par\smallskip}

\begin{corollary}
\label{cor:type-compatibility}
For types $\theta=(C,\lambda)$ and $\phi=(D,\mu)$,
\[
 E_\theta\cap E_\phi=\varnothing
 \quad\Longleftrightarrow\quad
 \lambda|_{\Sset_C\cap \Sset_D}\neq\mu|_{\Sset_C\cap \Sset_D}.
\]
In particular, if $\Sset_C\cap \Sset_D=1$, then $E_\theta\cap E_\phi\neq\varnothing$ for every pair of types based on $C$ and $D$.
\end{corollary}

\begin{proof}
The first statement is \cref{prop:extensions-intersection}.  If
$\Sset_C\cap \Sset_D=1$, all characters have the same restriction to the trivial
subgroup.
\end{proof}

\draft{\par\smallskip\noindent\textbf{\Cref{cor:type-compatibility}.}
Apply \cref{prop:extensions-intersection} with
$S=\Sset_C$ and $T=\Sset_D$.  The intersection of extension sets is empty
exactly when the restrictions disagree.  If $S\cap T=1$, the trivial group
has only its trivial character, so both restrictions necessarily agree and
the two extension sets meet.\par\smallskip}

\subsection{Finite Fourier transform}

If $Y\subseteq X$, write $\mathbf 1_Y\in\Fun(X,\R)$ for its characteristic
function:
\[
 \mathbf 1_Y(x)=\begin{cases}1_\R,&x\in Y,\\0,&x\notin Y.\end{cases}
\]
For $f\in\Fun(X,\R)$, its support is $\{x\in X:f(x)\neq0\}$.  We use the unnormalized Fourier transform
$\Fourier:\Fun(\hQ,\R)\to\Fun(\Q,\R)$ defined by
$\Fourier(f)(q)=\sum_{\xi\in\hQ}f(\xi)\xi(q)$.
By \eqref{eq:character-orthogonality}, the character table of $\Q$ is
invertible over $\R$.  Hence $\Fourier$ is an isomorphism, with inverse
\[
 \Fourier^{-1}(a)(\xi)
 =\frac{1}{|\Q|}\sum_{q\in \Q}a(q)\xi(q^{-1}).
\]

\draft{\par\smallskip\noindent\textbf{Checking the transform and its inverse.}
If functions are written as column vectors indexed by $\hQ$, then
$\Fourier$ is multiplication by the character table
$X=(\xi(q))_{q,\xi}$.  Character orthogonality says that the matrix with
entries $|\Q|^{-1}\xi(q^{-1})$ is its inverse.  This gives exactly the displayed inverse formula.  The factor $|\Q|^{-1}$ exists in $\R$ by
\Cref{lem:split-stabilizers}.  No field hypothesis is present here.\par\smallskip}

\begin{proposition}
\label{prop:Fourier-extensions}
For $S\leq \Q$ and $\lambda\in\widehat S$,
\[
 \Fourier(\mathbf 1_{E(S,\lambda)})
 =[\Q:S]\,\lambda\,\mathbf 1_S,
\]
where $\lambda\mathbf 1_S$ is extended by zero outside $S$.
\end{proposition}

\begin{proof}
Choose an extension $\widetilde\lambda\in\hQ$.  Then
$E(S,\lambda)=\widetilde\lambda S^\perp$, where
$S^\perp=\{\eta\in\hQ:\eta|_S=1\}$.
Therefore
\[
 \Fourier(\mathbf 1_{E(S,\lambda)})(q)
 =\widetilde\lambda(q)
   \sum_{\eta\in S^\perp}\eta(q).
\]
The map $\eta\mapsto(qS\mapsto\eta(q))$ identifies $S^\perp$ with
$\widehat{\Q/S}$.  Hence character orthogonality in $\Q/S$ gives
\[
 \sum_{\eta\in S^\perp}\eta(q)
 =\begin{cases}
   [\Q:S],&q\in S,\\
   0,&q\notin S.
  \end{cases}
\]
For $q\in S$, $\widetilde\lambda(q)=\lambda(q)$, which proves the formula.
\end{proof}

\draft{\par\smallskip\noindent\textbf{\Cref{prop:Fourier-extensions}, expanded.}
Fix one extension $\widetilde\lambda$.  Every other extension differs from
it by a character trivial on $S$, and conversely multiplying by such a
character does not change the restriction.  Hence
$E(S,\lambda)=\widetilde\lambda S^\perp$.

A character $\eta\in S^\perp$ is constant on cosets of $S$, so it defines a
character of $\Q/S$.  Every quotient character inflates in this way.  Thus
$S^\perp\cong\widehat{\Q/S}$.  The sum
$\sum_{\eta\in S^\perp}\eta(q)$ is the full character sum of $\Q/S$ at the
coset $qS$.  It is $[\Q:S]$ when $qS=S$, namely $q\in S$, and zero otherwise.
Multiplication by $\widetilde\lambda(q)$ gives the stated formula, and on
$S$ this scalar is $\lambda(q)$.\par\smallskip}

For $S\leq \Q$, let
\[
 \mV_S
 =\Span_\R
  \{\mathbf 1_{E(S,\lambda)}:\lambda\in\widehat S\}
 \subseteq\Fun(\hQ,\R).
\]
For any subset $X\subseteq \Q$, let $\Fun_X(\Q,\R)$ denote the
functions on $\Q$ supported in $X$.

\begin{theorem}
\label{thm:Fourier-VS}
For every subgroup $S\leq \Q$,
\[
 \Fourier(\mV_S)=\Fun_S(\Q,\R).
\]
Consequently, for subgroups $S,T\leq \Q$,
\begin{equation}
 \mV_S\cap\mV_T=\mV_{S\cap T},
 \label{eq:VS-intersection}
\end{equation}
and for any finite family $\Sset$ of subgroups of $\Q$,
\begin{equation}
 \Fourier\left(\sum_{S\in\Sset}\mV_S\right)
 =\Fun_{\cup_{S\in\Sset}S}(\Q,\R).
 \label{eq:VS-sum}
\end{equation}
In particular,
\begin{equation}
 \operatorname{rank}_\R\left(\sum_{S\in\Sset}\mV_S\right)
 =\left|\bigcup_{S\in\Sset}S\right|.
 \label{eq:VS-rank}
\end{equation}
\end{theorem}

\begin{proof}
By \cref{prop:Fourier-extensions}, $\Fourier(\mV_S)$ is spanned by
$[\Q:S]\lambda\mathbf 1_S$, $\lambda\in\widehat S$.  Since $[\Q:S]\in\uR$
and the characters of $S$ form an $\R$-basis of $\Fun(S,\R)$, this is
$\Fun_S(\Q,\R)$.  Therefore
\[
 \Fourier(\mV_S\cap\mV_T)
 =\Fun_S(\Q,\R)\cap\Fun_T(\Q,\R)
 =\Fun_{S\cap T}(\Q,\R),
\]
which gives \eqref{eq:VS-intersection}.  For a finite family
$\Sset=\{S_1,\ldots,S_r\}$ and a function $f$ supported on their union, put
$f_i=f\mathbf 1_{S_i\setminus(S_1\cup\cdots\cup S_{i-1})}$.  Then
$f=f_1+\cdots+f_r$ and each $f_i$ is supported on $S_i$.  Hence
\[
 \sum_{S\in\Sset}\Fun_S(\Q,\R)
 =\Fun_{\cup_{S\in\Sset}S}(\Q,\R).
\]
Applying $\Fourier^{-1}$ gives \eqref{eq:VS-sum}.  The module on the
right is free with basis indexed by $\bigcup_{S\in\Sset}S$, proving
\eqref{eq:VS-rank}.
\end{proof}

\draft{\par\smallskip\noindent\textbf{\Cref{thm:Fourier-VS}, expanded.}
\Cref{prop:Fourier-extensions} sends the generators of $\mV_S$ to
$[\Q:S]\lambda\mathbf1_S$.  The scalar $[\Q:S]$ is a unit and the characters
$\lambda$ form a basis of all functions on $S$, so these images span exactly
the functions supported on $S$.

Because $\Fourier$ is injective and surjective, it preserves intersections:
$\Fourier(U\cap V)=\Fourier(U)\cap\Fourier(V)$.  A function is supported in
both $S$ and $T$ exactly when it is supported in $S\cap T$, proving the
intersection formula.

For the sum formula, the successive differences in the proof assign every
point of the union to the first subgroup containing it.  Thus every supported
function is written as a sum of functions supported in individual subgroups.
Finally $\Fun_X(\Q,\R)$ is free with the delta functions at points of $X$ as
a basis, so its rank is $|X|$.\par\smallskip}

\begin{proposition}
\label{prop:extension-partitions}
Let $S,T\leq \Q$, put $U=S\cap T$, and let $\nu\in\widehat U$.  Then
\[
 \sum_{\substack{\lambda\in\widehat S\\\lambda|_U=\nu}}
   \mathbf 1_{E(S,\lambda)}
 =\mathbf 1_{E(U,\nu)}
 =\sum_{\substack{\mu\in\widehat T\\\mu|_U=\nu}}
   \mathbf 1_{E(T,\mu)}.
\]
Both sums are disjoint partitions of $E(U,\nu)$.
\end{proposition}

\begin{proof}
For $\xi\in\hQ$, the only possible term on the left is
$\lambda=\xi|_S$, and it occurs exactly when $\xi|_U=\nu$.  Hence the left sum
is $\mathbf 1_{E(U,\nu)}$.  The same argument applies to $T$.
\end{proof}

\draft{\par\smallskip\noindent\textbf{\Cref{prop:extension-partitions}, pointwise check.}
Fix $\xi\in\hQ$.  Among all characters of $S$, only
$\lambda=\xi|_S$ can satisfy $\xi\in E(S,\lambda)$.  This term is included
in the sum exactly when $\xi|_U=\nu$.  Hence the left side takes the value $1$ exactly on
$E(U,\nu)$ and $0$ elsewhere.  Since at most one term occurs at each point,
the sum is a disjoint partition.  The argument for $T$ is identical.\par\smallskip}

\begin{corollary}
\label{cor:same-Peirce}
Suppose $C,D\in \H\backslash \G/\K$ have stabilizers $S$ and $T$, respectively,
and put $U=S\cap T$.  For $\nu\in\widehat U$, the two graded bimodules
\[
 \displaystyle\bigoplus_{\substack{\lambda\in\widehat S\\\lambda|_U=\nu}}
 \mA_{(C,\lambda)} \quad\text{and}\quad
 \displaystyle\bigoplus_{\substack{\mu\in\widehat T\\\mu|_U=\nu}}
 \mA_{(D,\mu)}
\]
have isomorphic Peirce components for every $(\chi,\rho)$.  In particular their
underlying $(\RH,\RK)$-bimodules are isomorphic.  If $C\neq D$, they are not isomorphic as
graded bimodules.
\end{corollary}

\begin{proof}
By \cref{prop:extension-partitions,thm:Fourier-Peirce-formula}, both bimodules
have Peirce component $\R$ for $\xi\in E(U,\nu)$ and zero otherwise.  Since every ungraded bimodule is the direct sum of its Peirce components and
the $\Q$-action on the $\xi$-component is scalar through $\xi$, choose an
$\R$-linear isomorphism on each nonzero matching component and take their
direct sum.  The preceding observation shows that this direct sum is an
$(\RH,\RK)$-bimodule isomorphism.  If $C\neq D$, their
graded supports are the distinct double cosets $C$ and $D$, so they are not
graded isomorphic.
\end{proof}

\draft{\par\smallskip\noindent\textbf{\Cref{cor:same-Peirce}, why ungraded isomorphism follows.}
The partition identity says pointwise that each displayed bimodule has one
copy of $\R$ in the $\xi$-Peirce component for
$\xi\in E(U,\nu)$ and zero elsewhere.  Choose an $\R$-module isomorphism
between each pair of matching copies.  Since $q\in\Q$ acts on both copies by
the same scalar $\xi(q)$, this map is automatically $\RQ$-linear.  Their
finite direct sum is therefore an ungraded bimodule isomorphism.  It need not
preserve the $\G$-grading: when $C\neq D$, the two supports are disjoint
double cosets, and strict graded isomorphism is impossible.\par\smallskip}

Suppose that $\mM$ is a finite direct sum of atomic bimodules and write
\[
 \mP_{C,\lambda}(\mM)\cong \R^{m_{C,\lambda}}.
\]
The Peirce multiplicity function is integer-valued.  To apply the
$\R$-valued Fourier transform, for $\xi=(\chi,\rho)$ let
\[
 \overline n_{\mM}(\xi)
 =\operatorname{rank}_\R(\bee_\chi \mM \bee_\rho)\,1_\R
 \in\R
\]
be its image under $\mathbb Z\to\R$.  Then
\cref{thm:Fourier-Peirce-formula} gives
\[
 \overline n_{\mM}
 =\sum_{C,\lambda}
   (m_{C,\lambda}1_\R)\mathbf 1_{E(\Sset_C,\lambda)}.
\]
Applying \cref{prop:Fourier-extensions} gives
\begin{equation}
 \Fourier(\overline n_{\mM})(q)
 =\sum_{\substack{C\in \H\backslash \G/\K\\q\in \Sset_C}}
  [\Q:\Sset_C]
  \sum_{\lambda\in\widehat{\Sset_C}}
       (m_{C,\lambda}1_\R)\lambda(q).
 \label{eq:Peirce-multiplicity-Fourier}
\end{equation}
Thus every summand on the right of \eqref{eq:Peirce-multiplicity-Fourier} is
supported on a stabilizer $\Sset_C$.  In positive characteristic this Fourier
formula concerns the images of the integer multiplicities in $\R$.  It does not,
by itself, recover the integer multiplicities.

\draft{\par\smallskip\noindent\textbf{The multiplicity formula and positive characteristic.}
For an atomic summand of type $(C,\lambda)$, \cref{prop:Peirce-induced}
contributes one free copy of $\R$ at each $\xi\in E(\Sset_C,\lambda)$ and
nothing elsewhere.  With $m_{C,\lambda}$ copies, ranks add to the integer
function
\[
 n_{\mM}(\xi)=\sum_{C,\lambda}
 m_{C,\lambda}\mathbf1_{E(\Sset_C,\lambda)}(\xi).
\]
The Fourier transform in this section accepts values in $\R$, so the black
text first applies the canonical map $\mathbb Z\to\R$.  After that map,
linearity and \cref{prop:Fourier-extensions} give
\eqref{eq:Peirce-multiplicity-Fourier} term by term.  If $\R$ has
characteristic $p$, integers differing by a multiple of $p$ have the same
image in $\R$.  Hence the transformed $\R$-valued function cannot in general
recover the original integer ranks.\par\smallskip}

\section{Multiplicity-free graded bimodules}
\label{sec:multiplicity-free}

\begin{definition}
A finitely generated graded bimodule $\mM$ is called \emph{multiplicity-free}
if every nonzero Peirce component
$\bee_\chi\mM\bee_\rho$ is free of rank one over $\R$.
\end{definition}

\begin{theorem}
\label{thm:multiplicity-free}
A finitely generated graded bimodule $\mM$ is multiplicity-free if and only if
there is a finite set of types $\Theta\subseteq\Tset_{\H,\K}(\G)$ such that
\begin{equation}
 \mM\cong_{\gr}\bigoplus_{\theta\in\Theta}\mA_\theta
 \label{eq:multiplicity-free-decomposition}
\end{equation}
and, for distinct $\theta,\phi\in\Theta$,
\begin{equation}
 E_\theta\cap E_\phi=\varnothing.
 \label{eq:multiplicity-free-disjointness}
\end{equation}
\end{theorem}

\begin{proof}
Assume $\mM$ is multiplicity-free and let
$0\neq\mP_\theta(\mM)$, where $\theta=(C,\lambda)$.  Choose
$\xi\in E_\theta$.  By \cref{thm:Fourier-Peirce-formula},
$\mP_\theta(\mM)$ is an $\R$-direct summand of the nonzero Peirce component
$\bee_\xi\mM$.  This component is isomorphic to $\R$.  If $\R=N\oplus N'$ as $\R$-modules and $p:\R\to N$ is the projection, then
$e=p(1)$ is idempotent and $N=e\R$.  Thus direct summands of $\R$ are exactly
ideals of the form $e\R$.  Since $\R$ is indecomposable, a nonzero direct summand is $\R$ itself and its complement is
zero.  Hence $\mP_\theta(\mM)\cong\R$, and no other nonzero
$\mP_\phi(\mM)$ can satisfy $\xi\in E_\phi$.  Since $\mM$ is finitely
generated, only finitely many $\mP_\theta(\mM)$ are nonzero by
\cref{cor:finite-generation}.  This proves
\eqref{eq:multiplicity-free-decomposition} and \eqref{eq:multiplicity-free-disjointness}.

Conversely, if \eqref{eq:multiplicity-free-decomposition} holds and the sets $E_\theta$
are pairwise disjoint, \cref{thm:Fourier-Peirce-formula} shows that every
Peirce component is either zero or isomorphic to $\R$.  It follows from
\cref{cor:type-compatibility} that \eqref{eq:multiplicity-free-disjointness}
holds precisely when, for distinct $\theta=(C,\lambda)$ and $\phi=(D,\mu)$ in $\Theta$,
\[
 \lambda|_{\Sset_C\cap \Sset_D}\neq\mu|_{\Sset_C\cap \Sset_D}.
\]
\end{proof}

\draft{\par\smallskip\noindent\textbf{\Cref{thm:multiplicity-free}, expanded.}
Suppose $\mP_\theta(\mM)\neq0$.  An extension
$\xi\in E_\theta$ exists.  \Cref{thm:Fourier-Peirce-formula} places this coefficient module
as a direct summand of $\bee_\xi\mM\cong\R$.

If $\R=N\oplus N'$ and $p:\R\to N$ is the projection, put $e=p(1)$.  Then
$p(r)=re$, so $N=e\R$, and $p^2=p$ gives $e^2=e$.  Indecomposability forces
$e=0$ or $1$.  A nonzero direct summand is therefore all of $\R$, and its
complement is zero.  Hence $\mP_\theta(\mM)\cong\R$, and no second nonzero
coefficient module can contribute to the same $\xi$.  This is pairwise
disjointness of the $E_\theta$.  Finite generation makes the set of nonzero
types finite.

Conversely, disjointness means that a fixed Peirce component contains at most
one coefficient module, and each such module is $\R$.  Thus it is zero or
free of rank one.  The zero bimodule corresponds to $\Theta=\varnothing$.
The restriction condition is exactly \cref{cor:type-compatibility}.\par\smallskip}

When all stabilizers coincide, this recovers the distinct-character condition
of \cite[Theorem~2]{SSY2020}.  In general the characters must differ on the
corresponding stabilizer intersections.

\subsection{Gradings on a prescribed multiplicity-free bimodule}

For $\xi=(\chi,\rho)\in\hQ$, let $\R_\xi$ denote the rank-one ungraded
$(\RH,\RK)$-bimodule $\R$ with
\[
 h\cdot r\cdot k=\chi(h)\rho(k)r.
\]
If $B\subseteq\hQ$, put $\mM_B=\bigoplus_{\xi\in B}\R_\xi$.

\draft{\par\smallskip\noindent\textbf{The prescribed ungraded bimodule.}
The module $\R_\xi$ has one nonzero Peirce component: its $\xi$-component is
$\R$.  Indeed $\bee_\alpha$ acts as $1$ exactly when $\alpha=\chi$, and the
same holds on the right for $\rho$.  Therefore $\mM_B$ has Peirce component
$\R$ at every $\xi\in B$ and zero elsewhere.\par\smallskip}

\begin{theorem}
\label{thm:exact-cover}
Graded isomorphism classes of multiplicity-free $\G$-gradings on the
underlying $(\RH,\RK)$-bimodule $\mM_B$ are in bijection with finite subsets
$\Theta\subseteq\Tset_{\H,\K}(\G)$ satisfying
\begin{equation}
 B=\mathop{\dot\bigcup}_{\theta\in\Theta}E_\theta.
 \label{eq:exact-cover}
\end{equation}
The double coset is part of the type: if two distinct types have the same set $E_\theta$, they still determine nonisomorphic graded atomic bimodules.
\end{theorem}

\begin{proof}
By \cref{thm:multiplicity-free}, a multiplicity-free graded bimodule is
$\bigoplus_{\theta\in\Theta}\mA_\theta$ with the sets $E_\theta$ pairwise
disjoint.  By \cref{thm:Fourier-Peirce-formula}, its underlying ungraded
bimodule has nonzero Peirce components precisely at
$\bigcup_{\theta\in\Theta}E_\theta$, each free of rank one.  The Peirce decomposition displayed at the beginning of
\Cref{sec:Fourier-Peirce} therefore shows that it is isomorphic to
$\mM_B$ exactly when \eqref{eq:exact-cover} holds.  If
$f:\mM_B\to\bigoplus_{\theta\in\Theta}\mA_\theta$ is such an ungraded
bimodule isomorphism, define
$(\mM_B)^g=f^{-1}((\bigoplus_\theta\mA_\theta)^g)$.  This transports the grading
to $\mM_B$.  Finally, \cref{thm:classification} shows that two such
gradings are graded isomorphic exactly when their sets of types coincide, even
when distinct types determine the same set $E_\theta$.
\end{proof}

\draft{\par\smallskip\noindent\textbf{\Cref{thm:exact-cover}, expanded.}
Each atomic summand of type $\theta$ contributes one copy of $\R$ precisely
at the Peirce indices in $E_\theta$.  Pairwise disjointness prevents a rank
larger than one.  Hence the underlying bimodule has the same Peirce
decomposition as $\mM_B$ exactly when the disjoint union of these sets is
$B$.

An ungraded isomorphism $f$ transports the grading by inverse images:
\[
 \mM_B=\bigoplus_g f^{-1}(N^g).
\]
Bimodule linearity gives $f(hmk)=hf(m)k$, so the transported components obey
the grading rule.  For uniqueness, the classification records the full type
$(C,\lambda)$, not merely $E_\theta$.  Thus different double cosets still
give nonisomorphic strict gradings when the extension sets coincide.\par\smallskip}

\begin{corollary}
\label{cor:Fourier-support}
Let $\Sset_{\H,\K}(\G)=\{\Sset_C:C\in \H\backslash \G/\K\}$.  If $B\subseteq\hQ$ admits a multiplicity-free grading as in
\cref{thm:exact-cover}, then
\[
 \operatorname{supp}\Fourier(\mathbf 1_B)
 \subseteq
 \bigcup_{S\in\Sset_{\H,\K}(\G)}S.
\]
\end{corollary}

\begin{proof}
Since $\Q$ is finite, it has only finitely many subgroups, so
$\Sset_{\H,\K}(\G)$ is finite even when $\G$ is infinite.
Equation \eqref{eq:exact-cover} gives
$\mathbf 1_B\in\sum_{S\in\Sset_{\H,\K}(\G)}\mV_S$.
Apply \cref{thm:Fourier-VS}.
\end{proof}

\draft{\par\smallskip\noindent\textbf{\Cref{cor:Fourier-support}.}
The exact cover gives the pointwise identity
$\mathbf1_B=\sum_{\theta\in\Theta}\mathbf1_{E_\theta}$.  Each summand lies
in $\mV_{\Sset_C}$.  \Cref{thm:Fourier-VS} says that the Fourier
transform is supported on the union of those stabilizers, which is contained
in the union of all stabilizers for the pair $(\H,\K)$.\par\smallskip}

\begin{proposition}
\label{prop:two-subgroups}
Let $S,T\leq \Q$, put $U=S\cap T$, and let
$\Lambda\subseteq\widehat S$ and $\Mset\subseteq\widehat T$.  Suppose
\[
 B=\mathop{\dot\bigcup}_{\lambda\in\Lambda}E(S,\lambda)
  =\mathop{\dot\bigcup}_{\mu\in\Mset}E(T,\mu)
\]
for some $B\subseteq\hQ$.  Then, for every $\nu\in\widehat U$, either
$B\cap E(U,\nu)=\varnothing$ or $B\cap E(U,\nu)=E(U,\nu)$.
In the second case, $\Lambda$ contains every extension of $\nu$ to $S$ and
$\Mset$ contains every extension of $\nu$ to $T$.  Thus, on each
$E(U,\nu)$, the two decompositions are respectively
\begin{equation}
 \mathop{\dot\bigcup}_{\substack{\lambda\in\widehat S\\\lambda|_U=\nu}}E(S,\lambda)\quad\text{and}\quad
 \mathop{\dot\bigcup}_{\substack{\mu\in\widehat T\\\mu|_U=\nu}}E(T,\mu).
 \label{eq:two-subgroups-decompositions}
\end{equation}
\end{proposition}

\begin{proof}
Fix $\nu\in\widehat U$.  If no $\lambda\in\Lambda$ restricts to
$\nu$, then the first equality for $B$ and \cref{prop:extension-partitions}
give $B\cap E(U,\nu)=\varnothing$.

Suppose instead that $\lambda\in\Lambda$ and $\lambda|_U=\nu$.  Let
$\mu$ be any extension of $\nu$ to $T$.  By
\cref{prop:extensions-intersection},
$E(S,\lambda)\cap E(T,\mu)\neq\varnothing$.  Since
$E(S,\lambda)\subseteq B$ and the sets $E(T,\mu')$ form a partition of
$\hQ$, the second expression for $B$ forces $\mu\in\Mset$.  Thus every
extension of $\nu$ to $T$ belongs to $\Mset$.  Reversing the roles of $S$
and $T$ shows that every extension of $\nu$ to $S$ belongs to $\Lambda$.
Now \cref{prop:extension-partitions} gives
$B\cap E(U,\nu)=E(U,\nu)$ and the two decompositions in
\eqref{eq:two-subgroups-decompositions}.
\end{proof}

\draft{\par\smallskip\noindent\textbf{\Cref{prop:two-subgroups}, the forcing step.}
The full families $E(S,\lambda)$ and $E(T,\mu)$ each partition $\hQ$.
Suppose a selected $S$-part above $\nu$ occurs, and fix any extension $\mu$
of $\nu$ to $T$.  \Cref{prop:extensions-intersection} supplies
\[
 \xi\in E(S,\lambda)\cap E(T,\mu).
\]
The first membership puts $\xi$ in $B$.  In the full $T$-partition, the
unique part containing $\xi$ is $E(T,\mu)$.  Since $B$ is a union of selected
$T$-parts, this part must be selected.  Thus every extension of $\nu$ to $T$
belongs to $\Mset$.  Reversing $S$ and $T$ selects every extension to $S$.
\Cref{prop:extension-partitions} then makes both selected unions all of
$E(U,\nu)$.  This proves the all-or-nothing statement without invoking any
unstated coset-partition result.\par\smallskip}

\begin{remark}
\label{rem:coset-partitions}
Each $E(S,\lambda)$ is a coset of $S^\perp\leq\hQ$.  Hence a
multiplicity-free decomposition for which every Peirce component is nonzero
gives a coset partition of $\hQ$.  Akman and Sissokho \cite[Theorems~2 and~3]{AkmanSissokho2025}
treat partitions involving two subgroups and three mutually commuting
subgroups, respectively, and exhibit nonstandard examples involving four
subgroups.  \Cref{prop:two-subgroups} applies more generally to a
subset $B\subseteq\hQ$, with the subgroups $S^\perp$ arising from the
stabilizers considered above.
\end{remark}

\subsection{A nonabelian example}

\begin{example}
\label{ex:S4-V4}
Let $\F$ be a splitting field of characteristic different from $2$, let
$\G=S_4$ and $\H=\K=V_4$, where $V_4$ is the Klein four subgroup of double transpositions.  Conjugation
by $S_4$ permutes the three nonidentity elements of $V_4$.  Its kernel is
$V_4$, so
\[
 S_4/V_4\cong\Aut(V_4)\cong S_3.
\]
Thus the six double cosets $\H\backslash \G/\H$ may be indexed by
$\sigma\in S_3$.  If $g_\sigma$ induces the automorphism $\sigma$ of $V_4$,
then
\[
 \Sset_\sigma
 =\{(\sigma(k^{-1}),k):k\in V_4\}
 =\{(\sigma(k),k^{-1}):k\in V_4\}
 \leq V_4\times V_4
\]
is the graph of $\sigma$ and has order $4$.

For $\sigma\neq\tau$,
\begin{equation}
 |\Sset_\sigma\cap \Sset_\tau|
 =\begin{cases}
   2,&\sigma^{-1}\tau\text{ is a transposition},\\
   1,&\sigma^{-1}\tau\text{ is a $3$-cycle}.
  \end{cases}
 \label{eq:S4-stabilizer-intersections}
\end{equation}
Indeed, an element belongs to both stabilizers exactly when it has the form
$(\sigma(k),k^{-1})$ with $\sigma(k)=\tau(k)$.  Thus the intersection is
parametrized by the fixed points of $\sigma^{-1}\tau$ on $V_4$.  Since $|\Q|=16$ and $|\Sset_\sigma|=4$, every set
$E(\Sset_\sigma,\lambda)$ has four elements.

By \cref{cor:type-compatibility}, types based on two distinct stabilizers can
occur in the same multiplicity-free decomposition only when the stabilizers
intersect in order $2$.  Join two of the six stabilizers when this happens.
By \eqref{eq:S4-stabilizer-intersections}, the resulting graph is the Cayley
graph of $S_3$ with respect to its three transpositions.  Right multiplication
by a transposition joins each even permutation to each odd permutation, so this
graph is $K_{3,3}$.  In particular it has nine edges and no triangles.  The
stabilizers used in a multiplicity-free decomposition must be pairwise adjacent
by \cref{cor:type-compatibility}, hence form a clique.  Therefore a decomposition
for which every Peirce component is nonzero uses at most two distinct
stabilizers.

For each stabilizer $S$, the four sets $E(S,\lambda)$,
$\lambda\in\widehat S$, give one partition of $\hQ$.  This gives six
partitions using one stabilizer.  Now let
$\{\Sset_\sigma,\Sset_\tau\}$ be an edge and put
$U=\Sset_\sigma\cap\Sset_\tau$.  Since $|U|=2$, the two sets
$E(U,\nu)$, $\nu\in\widehat U$, partition $\hQ$.  By
\cref{prop:extension-partitions}, each $E(U,\nu)$ has the two decompositions
obtained from the extensions of $\nu$ to $\Sset_\sigma$ and to
$\Sset_\tau$.  Moreover, by \cref{prop:extensions-intersection}, every
$\Sset_\sigma$-part in the first decomposition meets every
$\Sset_\tau$-part in the second.  Hence a disjoint cover of $E(U,\nu)$
using parts from these two stabilizers cannot mix the two decompositions: it
must use all extensions to one stabilizer or all extensions to the other.
Thus the two values of $\nu$ give four partitions of $\hQ$.  Two use only one
stabilizer and two use both.  Therefore the nine edges give eighteen
additional partitions, and the total number is $6+2\cdot9=24$.
By \cref{thm:exact-cover}, these give $24$ pairwise nonisomorphic
multiplicity-free gradings on the same underlying bimodule $\mM_{\hQ}$.
\draft{\par\medskip\noindent\textbf{The $S_4,V_4$ enumeration in full detail.}
\begin{enumerate}[label=\textup{(\arabic*)}]
\item The subgroup
$V_4=\{1,(12)(34),(13)(24),(14)(23)\}$ is normal because conjugation only
relabels the four letters.  Conjugation permutes its three nonidentity
elements, giving a homomorphism
\[
 S_4\longrightarrow S_3\cong\Aut(V_4).
\]
The element
$(12)$ fixes $(12)(34)$ and swaps the other two nonidentity elements, so its
image is a transposition.  The element $(123)$ acts as the $3$-cycle computed
in the $A_4$ example.  These two images generate $S_3$, so the action is
surjective.  Its kernel has order $24/6=4$ and contains the abelian subgroup
$V_4$.  Hence the kernel is exactly $V_4$.  Since $V_4$ is normal, each double
coset is a single coset, and there are six.

\item If $g_\sigma$ induces $\sigma$, the equation $hg_\sigma k=g_\sigma$
forces
$h=g_\sigma k^{-1}g_\sigma^{-1}=\sigma(k^{-1})$.  Conversely this value of
$h$ works.  Thus the stabilizer is the graph displayed in the black text above.  Its
second projection is bijective, so it has four elements.

\item A pair $(x,k^{-1})$ lies in both graph stabilizers exactly when
$\sigma(k)=\tau(k)$.  Thus $k$ is fixed by $\sigma^{-1}\tau$.
Regard $V_4$ as the two-dimensional vector space over
$\mathbb F_2$.  A transposition in $S_3$ fixes the identity and one
nonidentity element, hence two elements total.  A $3$-cycle fixes only the
identity.  This proves \eqref{eq:S4-stabilizer-intersections}.

\item The graph joining pairs with intersection order two is the Cayley graph
of $S_3$ with its transpositions.  Multiplication by an odd permutation
changes parity.  From any even permutation, the three transpositions reach
all three odd permutations, so the graph is $K_{3,3}$.  It has
$3\cdot3=9$ edges and no triangle.  Distinct stabilizers appearing in one
multiplicity-free cover must be pairwise joined.  Hence at most two occur.

\item For one fixed stabilizer $S$, restriction
$\hQ\to\widehat S$ partitions $\hQ$ into the four fibers
$E(S,\lambda)$.  This gives one pure cover for each of the six stabilizers.

\item Fix an edge $\{S,T\}$ and put $U=S\cap T$.  Since $|U|=2$, there are
two characters $\nu$ of $U$.  For each $\nu$, exactly two characters of $S$
and two of $T$ extend it.  Their extension sets give two different
two-part decompositions of $E(U,\nu)$.

Every $S$-part meets every $T$-part because their defining characters agree
on $U$.  Therefore a disjoint cover of $E(U,\nu)$ cannot select a part from
both decompositions: such parts would overlap.  It must select both $S$-parts
or both $T$-parts.  The two values of $\nu$ can be assigned independently,
giving four covers: two pure covers and two mixed covers.  The pure covers
were already counted.  Thus each edge contributes exactly two new covers.

\item Different edges give different mixed covers because the two
stabilizers used by a mixed cover recover that edge.  Hence the number is
$6+9\cdot2=24$.  \Cref{thm:exact-cover} converts these distinct sets
of types into pairwise nonisomorphic strict graded structures on the same
ungraded bimodule.
\end{enumerate}\par\medskip}
\end{example}

\section{Further questions}
\label{sec:questions}

We ask three questions suggested by the classification.

\begin{enumerate}[label=\textup{(\arabic*)}]
\item \emph{Pointed Hopf algebras.}  For suitable pointed Hopf algebras in
place of $\RH$ and $\RK$, with compatible Hopf coactions replacing group
gradings, is there an analogue of the orbit and stabilizer classification
of \cref{thm:classification}?  In particular, what data replace double
cosets, stabilizer characters, and coefficient modules, and under which
hypotheses can one recover an explicit Peirce decomposition?  Any such
extension must account for the non-group-like structure of pointed Hopf
algebras without relying on semisimplicity, which is not assumed even for the
group algebras treated here.

\item \emph{Smash products.}  Using the smash-product framework in
\cite{CM84}, can the classification and Peirce formulas
be expressed in terms of smash products?  The relationship between graded
modules and smash-product modules is known.  Can the double-coset types, their coefficient modules,
and the multiplicity-free compatibility conditions be recognized
intrinsically in a smash-product model, and can this perspective clarify
composition of graded bimodules?

\item \emph{Categories and composition.}  The category of ordinary
$(\RH,\RK)$-bimodules is the category of left
$\R[\H\times\K^{\mathrm{op}}]$-modules.  The category considered here
additionally carries a grading by the $\Q$-set $\G$.  The classification
already identifies this graded category, up to the choice of double-coset
representatives, with a product of categories of stabilizer modules.
The relative tensor product of two atomic bimodules is computed in
\cite[Section~7.2]{HindiSchutzer2026Incidence}.  Can this formula be extended
to arbitrary coefficient modules and expressed through these equivalences
compatibly with associativity when varying the intermediate subgroup?
Can the Peirce and multiplicity-free conditions
be expressed categorically without choosing representatives or characters?
\end{enumerate}

\draft{\clearpage
\noindent\textbf{Final verification checklist for the whole paper.}
The following dependency order lets the reader certify the manuscript without
using a later result prematurely.
\begin{enumerate}[label=\textup{(\arabic*)}]
\item \Cref{lem:stabilizer-basic} uses only the action and abelianness of
$\H$ and $\K$.
\item \Cref{lem:split-stabilizers} uses only the standing splitting and
ring hypotheses.  The Fourier-idempotent discussion then uses that lemma.
\item \Cref{lem:orbit-induction} uses only the group action and tensor
balancing.  \Cref{thm:classification} uses this lemma and the Fourier
idempotents, and the finite-generation and projectivity corollaries use the
same theorem.
\item In \cref{sec:YD} the coefficient ring is a field.  \Cref{lem:YD-orbit} uses
the Yetter-Drinfel'd condition.  \Cref{thm:YD-essential-image} uses that lemma,
the equivalence in \cref{thm:classification} on one orbit, and the character restriction
calculation.  No Peirce or multiplicity-free result is used there.
\item \Cref{prop:Peirce-induced} uses the induced construction and
orthogonality.  \Cref{thm:Fourier-Peirce-formula} then uses
\cref{thm:classification}.  The extension-set and Fourier results use only the split
character theory already established.
\item \Cref{thm:multiplicity-free} uses \cref{thm:Fourier-Peirce-formula} and
indecomposability.  \Cref{thm:exact-cover} uses it and
\cref{thm:classification}.  \Cref{prop:two-subgroups} uses only the two extension-set
propositions.  The final example then uses these results in that order.
\end{enumerate}

For edge cases, check the following explicit outcomes.
\begin{itemize}
\item The zero bimodule corresponds to all coefficient modules zero and, in
\cref{thm:multiplicity-free}, to the empty set of types.
\item If a stabilizer is trivial, it has one character and every character of
$\Q$ extends it.  The formulas give $E(1,1)=\hQ$.
\item If $\H=1$ or $\K=1$, the same action and stabilizer formulas remain
valid.  If $\G$ is infinite, only the set of double cosets can be infinite.
The group $\Q$, all stabilizers, and all character sets remain finite.
\item Arbitrary coefficient modules are allowed in the classification.
Finite generation, projectivity, freeness, and finite-dimensionality are used
only where they are explicitly stated.
\item The integer multiplicities in \cref{sec:Fourier-Peirce} are not
identified with elements of $\R$ until the map $\mathbb Z\to\R$ is written.
\item Every occurrence of ``isomorphic'' in the classification of gradings
means strict graded isomorphism.  The only deliberately ungraded comparison is
\cref{cor:same-Peirce} and the underlying-bimodule part of
\cref{thm:exact-cover}.
\end{itemize}

The external facts left to standard references are the following.  An element of a
commutative ring is a unit precisely when it lies in no maximal ideal, a
finitely generated projective module over a local ring is free, a projective
quotient splits its defining short exact sequence, direct summands of
projectives are projective, tensor product commutes with direct sums, the
normal form theorem for free products, and the elementary subgroup facts
about $A_4$, $S_4$, and $V_4$ used in the examples.  Every argument specific
to this paper is expanded above.\par\medskip}

\bibliographystyle{abbrv}
\bibliography{bimodules}

@article{NRV1990,
 author = {N{\u{a}}st{\u{a}}sescu, C. and Raianu, S. and Van Oystaeyen, F.},
 title = {Modules graded by {G}-sets},
 fjournal = {Mathematische Zeitschrift},
 journal = {Math. Z.},
 issn = {0025-5874},
 volume = {203},
 number = {4},
 pages = {605--627},
 year = {1990},
 language = {English},
 doi = {10.1007/BF02570759},
 zbMATH = {4188027},
 Zbl = {0721.16025}
}

@article{SSY2020,
 author = {Santulo, jun., Ednei A. and Souza, Jonathan P. and Yasumura, Felipe Y.},
 title = {Group gradings on finite dimensional incidence algebras},
 fjournal = {Journal of Algebra},
 journal = {J. Algebra},
 issn = {0021-8693},
 volume = {544},
 pages = {302--328},
 year = {2020},
 language = {English},
 doi = {10.1016/j.jalgebra.2019.10.026},
 zbMATH = {7128761},
 Zbl = {1436.16051}
}

@article{DSY2025,
 author = {Dos Santos, Helen Samara and Yasumura, Felipe Yukihide},
 title = {Group gradings on finite-dimensional incidence algebras. {II}},
 fjournal = {Linear Algebra and its Applications},
 journal = {Linear Algebra Appl.},
 issn = {0024-3795},
 volume = {726},
 pages = {273--290},
 year = {2025},
 language = {English},
 doi = {10.1016/j.laa.2025.07.027},
 zbMATH = {8091361}
}

@article{deFrancaSviridova2024,
 author = {de Fran{\c{c}}a, Antonio and Sviridova, Irina},
 title = {Irreducible graded bimodules over algebras and a {Pierce} decomposition of the {Jacobson} radical},
 fjournal = {Communications in Algebra},
 journal = {Commun. Algebra},
 issn = {0092-7872},
 volume = {52},
 number = {10},
 pages = {4226--4254},
 year = {2024},
 language = {English},
 doi = {10.1080/00927872.2024.2343774},
 zbMATH = {7909486},
 Zbl = {1562.16002}
}

@article{AkmanSissokho2025,
 author = {Akman, Fusun and Sissokho, Papa A.},
 title = {Transversal coset partitions of groups},
 fjournal = {Beitr{\"a}ge zur Algebra und Geometrie},
 journal = {Beitr. Algebra Geom.},
 issn = {0138-4821},
 volume = {66},
 number = {2},
 pages = {417--441},
 year = {2025},
 language = {English},
 doi = {10.1007/s13366-024-00748-9},
 zbMATH = {8043095},
 Zbl = {1565.20061}
}

@misc{HindiSchutzer2026Incidence,
      title={A classification of group gradings on incidence algebras over commutative rings}, 
      author={Felipe de Mattos Chafik Hindi and Waldeck Schutzer},
      year={2026},
      eprint={2609.21184},
      note={\href{https://arxiv.org/abs/2609.21184}{arXiv:2609.21184}},
      archivePrefix={arXiv},
      primaryClass={math.RA},
      url={https://arxiv.org/abs/2609.21184}, 
}

@article{CM84,
 author = {Cohen, M. and Montgomery, S.},
 title = {Group-graded rings, smash products, and group actions},
 fjournal = {Transactions of the American Mathematical Society},
 journal = {Trans. Am. Math. Soc.},
 issn = {0002-9947},
 volume = {282},
 pages = {237--258},
 year = {1984},
 language = {English},
 doi = {10.2307/1999586},
 zbMATH = {3845773},
 Zbl = {0533.16001}
}

@book {Polcino-Sehgal,
    AUTHOR = {Polcino Milies, C\'{e}sar and Sehgal, Sudarshan K.},
     TITLE = {An introduction to group rings},
    SERIES = {Algebra and Applications},
    VOLUME = {1},
 PUBLISHER = {Kluwer Academic Publishers, Dordrecht},
      YEAR = {2002},
     PAGES = {xii+371},
      ISBN = {1-4020-0238-6},
   MRCLASS = {16S34 (16-01 16U60)},
  MRNUMBER = {1896125},
MRREVIEWER = {Edgar G. Goodaire},
       DOI = {10.1007/978-94-010-0405-3},
       URL = {https://doi.org/10.1007/978-94-010-0405-3},
}

\end{document}